\documentclass[11pt]{amsart}
\usepackage{amssymb,comment}
\usepackage{amsmath}
\usepackage{bbm}
\usepackage{graphicx} 
\usepackage[english]{babel}

\usepackage{float}
\usepackage{mathtools}
\usepackage{hyperref}
\usepackage[noadjust]{cite}
\usepackage[margin=3.2cm]{geometry}
\hypersetup{ 
    colorlinks=true,       
    linkcolor=blue,          
    citecolor=blue,        
    filecolor=blue,      
    urlcolor=blue           
}

\usepackage[pdf]{pstricks} 
\usepackage{epsfig}
\usepackage{pst-grad} 
\usepackage{pst-plot} 
\usepackage[space]{grffile} 
\usepackage{etoolbox} 
\makeatletter 
\patchcmd\Gread@eps{\@inputcheck#1 }{\@inputcheck"#1"\relax}{}{}
\makeatother

\usepackage[all]{xy}
\usepackage{mathrsfs}
\usepackage{graphics}
\usepackage{amsthm}

\theoremstyle{plain}\newtheorem{theorem}{Theorem}[section]\newtheorem{Theorem}{Theorem}\newtheorem{proposition}[theorem]{Proposition}\newtheorem{lemma}[theorem]{Lemma}\newtheorem{corollary}[theorem]{Corollary}

\theoremstyle{definition}\newtheorem{example}[theorem]{Example}\newtheorem{remark}[theorem]{Remark}

\newtheorem*{Notation}{Notation}

\def\Q{\mathbb{Q}}\def\C{\mathbb{C}}\def\N{\mathbb{N}}\def\Z{\mathbb{Z}}\def\R{\mathbb{R}}\def\K{\mathbb{K}}\def\F{\mathbb{F}}\def\DD{\mathbb{D}}\def\KK{\mathbb{K}}

\def\ZZ2{\mathbb{\Z/ 2\Z}}
\def\sb{\subset}\def\su{\subset}
\def\lb{\langle}\def\rb{\rangle}\def\ot{\otimes}\def\t{\times}

\def\c{\gamma}\def\v{\varphi}

\def\a{\alpha}\def\b{\beta}\def\d{\delta}\def\e{\epsilon}\def\s{\sigma}\def\De{\Delta}\def\La{\Lambda}\def\o{\omega}\def\la{\lambda}\def\la{\lambda}
\def\ov{\overline}\def\wt{\widetilde}
\def\DD{\mathbb{D}^2}

\def\ad{\text{ad}}\def\gl{\mathfrak{gl}}
\def\sl2{\mathfrak{sl}_2}
\def\su2{\mathfrak{su}(2)}
\def\Aut{\text{Aut}\,}
\def\id{\text{id}}
\def\Hom{\text{Hom}\,}

\def\deg{\text{deg}}

\def\TH3{\Theta_3^{H}}

\def\Uh{U_h(\sl2)}\def\Usl2{U_q(\sl2)}\def\usl2{\wt{U}_q(\sl2)}\def\Uqgl11{U_q(\mathfrak{gl}(1|1))}\def\Uqsl11{U_q(\mathfrak{sl}(1|1))}\def\Uh{U_h(\sl2)}\def\uq{\mathfrak{u}_q(\sl2)}

\def\Rep{\text{Rep}}

\def\II1{\text{II}_1}

\def\UHrqg{\overline{U}^H_q(\gg)}

\def\mod2{\ (mod \ 2)}

\def\kk{\mathbb{K}}\def\gl{\mathfrak{gl}}\def\gl11{\mathfrak{gl}(1|1)}

\def\gg{\mathfrak{g}}
\def\Aut{\text{Aut}}

\def\uDH{\underline{D(H)}}

\def\gg{\bold{g}}

\begin{document}

\title[Genus bounds from unrolled quantum groups at roots of unity]{Genus bounds from unrolled quantum groups at roots of unity}
\author{Daniel L\'opez Neumann and Roland van der Veen}
\email{dlopezne@indiana.edu, r.i.van.der.veen@rug.nl}

\maketitle

\def\wD{D}
\def\sl{\mathfrak{sl}}
\def\ovU{\overline{D}}
\def\uq{\overline{U}_q(\mathfrak{sl}_2)}
\def\UUq{\overline{U}_q(\mathfrak{sl}_2)}
\def\Hh{\mathcal{H}}\def\RR{\mathcal{R}}
\def\gg{\mathfrak{g}}
\def\k{k}
\def\K{k}
\def\ka{\k^{\a}}\def\kb{\k^{\b}}\def\kc{\k^{\c}}\def\Da{D_{\a}}
\def\Dp{D(p)}
\def\ko{\ov{k}}
\def\Hp{H_p}

\def\bold{\boldsymbol}
\def\kk{\bold{k}}
\def\ll{\bold{l}}
\def\nn{\bold{n}}
\def\ZZ{\mathcal{Z}}
\def\uDH{\underline{D(H)}}
\def\tt{\theta}

\def\uq{\mathfrak{u}_{\zeta}(\sl_2)}

\def\gg{\mathfrak{g}}
\def\Uqag{U_{q,\a}(\gg)}
\def\qi{\qi}\def\di{d_i}

\def\K{K}\def\hh{\mathfrak{h}}

\def\Ur{\ov{U}_q(\gg)}
\def\UHqg{U^H_{q}(\gg)}\def\Uh{U(\mathfrak{h})}\def\Uqg{U_q(\gg)}
\def\UHr{\ov{U}^H_{q}(\gg)}\def\Ur{\ov{U}_{q}(\gg)}
\def\UHsr{\widetilde{U}^H_{q}(\gg)}
\def\RR{\mathcal{R}}
\def\Ua{\ov{U}_{\a}}\def\Ub{\ov{U}_{\b}}

\def\Uqg{U_q(\gg)}

\def\Bqg{\ov{B}_q(\gg)}\def\hh{\mathfrak{h}}\def\CC{\mathcal{C}^H}
\def\uqg{\mathfrak{u}_q(\gg)}

\def\uu{\mathfrak{u}}
\def\uqb{\uu_q(\mathfrak{b})}

\def\Dep{\De^+}
\def\ba{\boldsymbol{a}}
\def\mm{\boldsymbol{m}}

\def\ADO{\text{ADO}}\def\td{\text{d}}
\def\ADOgr{\text{ADO}_{\gg,r}}

\begin{abstract}
    For any simple complex Lie algebra $\gg$, we show that the degrees of the ``ADO" link polynomials coming from the unrolled restricted quantum group $\UHrqg$ at a root of unity give lower bounds to the Seifert genus of the link. We give a direct simple proof of this fact relying on a Seifert surface formula involving universal $\uqg$-invariants, where $\uqg$ is the small quantum group. We give a second proof by showing that the invariant $P_{\uqb}^{\theta}(K)$ of our previous work \cite{LNV:genus} coincides with such ADO invariants, where $\uqb$ is the Borel part of $\uqg$. To prove this, we show that equivariantizations of relative Drinfeld centers of crossed products essentially contain unrolled restricted quantum groups, a fact that could be of independent interest.
\end{abstract}

\section{Introduction}

\def\BB{B_{\zeta_p}(\gg)}



The quantum invariants of knots and links form a general class of topological invariants built from the theory of braided tensor categories, the main source being the representation categories of quantum groups $\Uqg$, which depend on a parameter $q\in\C$. These categories are semisimple if and only if $q$ is not a root of unity, and this fact leads to two different families of link polynomials. The generic case leads to the well-known Jones, HOMFLY and Kauffman polynomials and their colored versions. When $q$ is a root of unity, one obtains the Akutsu-Deguchi-Ohtsuki (ADO) invariants \cite{ADO} or the Costantino-Geer-Patureau (CGP) invariants \cite{CGP:non-semisimple}. While the former family has been intensively studied since their introduction, interest in the latter family of ``non-semisimple" invariants seems to have grown only quite recently as they appear in a new kind of TQFTs \cite{BCGP} and are conjectured to be related to the $q$-series predicted by physicists \cite{GM:two-variable, GPPV:BPS, Gukov:Coulomb}.\\

In our previous work \cite{LNV:genus}, we found a structural difference between the above two families of invariants. While relations between quantum invariants at generic $q$ and geometric topology are at most the subject of various conjectures \cite{Kashaev:hyperbolic, Garoufalidis:character}, we showed that the ADO invariants of \cite{ADO} (for $\gg=\sl_2$) provide lower bounds to the Seifert genus. This was part of a general theorem giving genus bounds from finite dimensional Hopf algebras. Previous to our work, the only genus bound known to come from quantum invariants (and that is not simply the Alexander bound) was that of Ohtsuki for the 2-loop part of the Kontsevich invariant \cite{Ohtsuki:2loop}. After our work, genus bounds for link polynomials coming from the quantum supergroup of $\sl(2|1)$ (a.k.a. Links-Gould invariants) were found in \cite{KT:Links-Gould}.\\

We now explain our results in detail. Let $\gg$ be a simple complex Lie algebra of rank $n$. Let $\hh$ be a Cartan subalgebra, $\a_1,\dots,\a_n$ a set of simple roots and let $\Dep$ be the set of positive roots. Write $$\sum_{\b\in\Dep}\b=k_1\a_1+\dots+k_n\a_n$$ where $k_1,\dots,k_n\in\N$. Let $q$ be a primitive root of unity of order $r\geq 2$ and let $r'=r$ if $r$ is odd and $r'=r/2$ if $r$ is even. We further suppose $r$ coprime to the determinant of the Cartan matrix. If $f$ is a Laurent polynomial in variables $t_1,\dots,t_n$ (possibly with half-integer powers) and if $N_i$ (resp. $n_i$) denotes the maximal power (resp. minimal power) of $t_i$ appearing in some monomial in $f$, then we define $\deg_{t_i}f:=N_i-n_i$.

\bigskip

\begin{Theorem}
\label{theorem: ADO genus}
Suppose $\gg$ is simply-laced. If $L\sb S^3$ is an $s$-component link, then there is an isotopy invariant $$\ADOgr'(L)\in (t_1^{k_1}\cdots t_n^{k_n})^{\frac{(r'-1)(s-1)}{2}}\cdot \Q(q)[t_1^{\pm 1},\dots,t_n^{\pm 1}] $$
    satisfying 
    $$\deg_{t_i} \ADOgr'(L)\leq (2g(L)+s-1)(r'-1)k_i$$
    
  \noindent  for every $i=1,\dots,n$, where $g(L)$ is the Seifert genus of $L$.
\end{Theorem}

\bigskip

The definition of $\ADOgr'(L)$ is based on the category $\CC$ of weight modules over the unrolled restricted quantum group $\UHrqg$ at a root of unity, which is a ribbon category \cite{GP:nonrestricted, GP:trace-projective-quantum-groups, Rupert:unrolled}. In this case, there are finite-dimensional weight modules $V_{\la}$ with highest weight $\la$ for all $\la\in\hh^*$ and $\ADOgr'(L)|_{t_i=q^{2\la(H_i)}}$ is the result of applying the Reshetikhin-Turaev construction \cite{RT1} to any $(1,1)$-tangle with closure $L$ in which all components have the same color $V_{\la}$. We denote it by $\ADO'$, since the original ADO \cite{ADO} and generalizations \cite{CGP:non-semisimple} are certain normalizations of ours. Our proof is based on a Seifert surface formula involving the action of a universal $\uqg$-invariant $Z_T$ on $V_{\la}\ot V_{\la}^*$, where we can apply the degree bound of Proposition \ref{lemma: ESSENTIAL LEMMA}. For the non-simply-laced case see Theorem \ref{theorem: non-simply-laced Thm 1}.\\

The invariant $\ADO'_{\sl_2,4}$ is the 1-variable Alexander polynomial $\De_L$ \cite{Murakami:Alexander} and the genus bound for $\ADO'_{\sl_2,r}(K)$ was shown in \cite{LNV:genus}. The invariant $\ADO'_{\sl_3,4}$ was studied by Harper in \cite{Harper:sl3-invariant} and shown to distinguish the Kinoshita-Terasaka from the Conway knot. As an example, we show that $\ADO'_{\sl_3,4}$ actually gives the correct genus bounds for these two knots.\\



\def\CCC{\wt{\mathcal{C}}^H}
In our second theorem, we explain the above bound through our previous work. Based on \cite{LN:TDD}, we built in \cite{LNV:genus} an $n$-variable knot polynomial $P_H^{\theta}(K)$ out of a $\N^n$-graded Hopf algebra $H$ of finite dimension and we showed that the total degree $\deg \ P_H^{\theta}$ gives lower bounds to the Seifert genus. This construction relied on the relative Drinfeld center of a crossed product $\Rep(H)\rtimes \Aut(H)$, or equivalently, the twisted Drinfeld double of Virelizier \cite{Virelizier:Graded-QG}. When $H=\uqb$, the Borel part of the small quantum group $\uqg$, there is an action of $\hh^*\cong\C^n$ on $\Rep(\uqb)$ and one gets a braided $\hh^*$-crossed category $\ZZ_{\Rep(\uqb)}(\Rep(\uqb)\rtimes \hh^*)$. The $\hh^*$-equivariantization of $\ZZ_{\Rep(\uqb)}$ is a braided category in the usual sense. Kauffman-Radford pairs on $\uqb$ (see Subsection \ref{subs: ribbon for TDDs}) induce ribbon structures on the equivariantization. Let $(\CCC, c^{\chi})$ be the ribbon category of Subsection \ref{subs: AJ and unrolled}, this is just a slight variation of $\CC$.

\bigskip

\begin{Theorem}
\label{theorem: equivariantization}
The category $(\CCC,c^{\chi})$ embeds as a full ribbon subcategory of the $\hh^*$-equivarianti-
zation of $\ZZ_{\Rep(\uqb)}(\Rep(\uqb)\rtimes\hh^*)$. Here, the equivariantization is endowed with a ribbon structure coming from a canonical Kauffman-Radford pair on $\uqb$.
\end{Theorem}

\bigskip

We find interesting that the braided and ribbon structure of $\CC$ can be deduced from that of the small quantum group by the above theorem. Relative Drinfeld centers appear in Turaev-Virelizier's work on HQFT \cite{TV:graded-center}, so our theorem might be useful in order to define TQFTs for ADO or CGP invariants in a Turaev-Viro style. From the above theorem we deduce the following (which was shown in \cite{LNV:genus} for $\sl_2$).

\begin{Theorem}
\label{corollary: P coincides with ADO}
    The multivariable polynomial invariant $P_{\uqb}^{\theta}(K)$ of \cite{LNV:genus} equals $\ADO'_{\gg,r}(K)$.
\end{Theorem}

Note that the genus bound of \cite{LNV:genus} was only for the total degree and only for knots in $S^3$, so Theorem \ref{theorem: ADO genus} is still more general that what can be deduced via Theorem \ref{corollary: P coincides with ADO} and \cite{LNV:genus}.\\

The plan of the paper is the following. Section \ref{section: unrolled} contains some preliminaries on unrolled quantum groups and weight modules. The main result here is the degree bound of Proposition \ref{lemma: ESSENTIAL LEMMA}. In Section \ref{section: link invariants} we define the $\ADOgr'$ invariants of links, give a Seifert surface formula for them in terms of universal $\uqg$-invariants and prove Theorem \ref{theorem: ADO genus}. We also give some computations in the case $\gg=\sl_3$ and $r=4$. Finally, in Section \ref{section: equivariantization thm} we define crossed products, relative Drinfeld centers, equivariantization and we prove Theorems \ref{theorem: equivariantization} and \ref{corollary: P coincides with ADO}.

\def\qi{q}

\begin{Notation}
During the whole paper, $\gg$ denotes a complex simple Lie algebra of rank $n$, $\hh$ a Cartan subalgebra, $\{\a_1,\dots,\a_n\}$ a set of simple roots, $(a_{ij})$ its Cartan matrix and $\Dep$ the set of positive roots. For simplicity, we will work in the simply-laced case, so $(a_{ij})$ is symmetric (the non-simply-laced case will be treated in Subsection \ref{subs: non-simply-laced}). We define $H_1,\dots,H_n\in \hh$ by $\a_i(H_j)=a_{ij}$ for all $i,j$. We let $(\a_i,\a_j)=a_{ij}$ and extend it to a form $(-,-)$ on $\hh^*$. We denote by $L_R, L_W\sb \hh^*$ the root and weight lattices of $\gg$ respectively. We denote $\rho=\frac{1}{2}\sum_{\b\in\Dep}\b\in L_W$. We also denote by $L_R^+$ the set of positive vectors in $L_R$ and for any $\c_1,\c_2\in L_R$ we write $\c_1> \c_2$ if $\c_1-\c_2\in L_R^+.$\\

Let $q$ be a primitive root of unity of order $r$, say $q=e^{\frac{2\pi i}{r}}$, and for every $z\in \C$ set $q^z=e^{\frac{2\pi i z}{r}}$ and $[z]_{\qi}=\frac{\qi^z-\qi^{-z}}{\qi-\qi^{-1}}$. For $k\in\N$ denote $[k]_q!=[k]_q[k-1]_q\cdots [1]_q$. For simplicity, we will suppose that $r$ is coprime to $\det(a_{ij})$. This can be avoided by working in the weight lattice $L_W$, but we don't do this.   \\

\end{Notation}

\section{Unrolled restricted quantum groups}
\label{section: unrolled}

In this section we recall some standard facts from unrolled quantum groups: PBW basis, restricted quantum groups, weight modules, braiding, pivotal and ribbon structure, etc. We follow mainly \cite{GP:trace-projective-quantum-groups} and \cite{Rupert:unrolled}. The last two subsections are the build up for Proposition \ref{lemma: ESSENTIAL LEMMA}.

\subsection{Unrolled quantum groups} Let $\UHqg$ be the algebra with generators $E_i,F_i,K_i^{\pm 1}, H_i$, $i=1,\dots,n$ and relations
\begin{align*}
[H_i,E_j]&=a_{ij}E_j, & [H_i,F_j]&=-a_{ij}F_j, & [H_i,K_j^{\pm 1}]&=0=[H_i,H_j],\\
K_iE_j&=\qi^{a_{ij}}E_jK_i, & K_iF_j&=\qi^{-a_{ij}}F_jK_i & K_iK_j&=K_jK_i, \\
[E_i,F_j]&=\d_{ij}\frac{K_i-K_i^{-1}}{\qi-\qi^{-1}}, & K_iK_i^{-1}&=K_i^{-1}K_i=1, & & 
\end{align*}
for every $i,j$ and
\begin{align*}
  \sum_{k=0}^{1-a_{ij}}(-1)^k\left[\begin{matrix} 1-a_{ij} \\ 
 k\end{matrix}\right]_{\qi}E_i^{1-a_{ij}-k}E_jE_i^k&=0,  \\
  \sum_{k=0}^{1-a_{ij}}(-1)^k\left[\begin{matrix} 1-a_{ij} \\ 
 k\end{matrix}\right]_{\qi}F_i^{1-a_{ij}-k}F_jF_i^k&=0, 
\end{align*}
for $i\neq j$. This is a Hopf algebra if we define 
\begin{align*}
\De(E_i)&=E_i\ot K_i+1\ot E_i, & \De(F_i)&=K_i^{-1}\ot F_i+F_i\ot 1, \\
 \De(K_i)&=K_i\ot K_i, & \De(H_i)&=H_i\ot 1+1\ot H_i
\end{align*}
for every $i$. For $\c=a_1\a_1+\dots+a_n\a_n\in L_R$ we denote $K_{\c}:=K_1^{a_1}\dots K_n^{a_n}$.

\medskip

We denote by $\Uqg\sb \UHqg$ the subalgebra generated by $E_i,F_i,K_i^{\pm 1}$. We denote by $\Uh$ the subalgebra of $\UHqg$ generated by the $H_i$'s. Let $V_+$ (resp. $V_-$) be the subalgebra generated by $E_1,\dots,E_n$ (resp. $F_1,\dots,F_n$).

\medskip

The algebra $\Uqg$ is $L_R$-graded if we set $d(E_i)=\a_i, d(F_i)=-\a_i$ and we extend by multiplicativity. This way $\Uqg$ becomes a $L_R$-graded Hopf algebra.

\subsection{PBW basis} 

Let $\Dep$ be the set of positive roots of $\gg$. The root vectors $E_1,\dots, E_n$ correspond to the simple roots $\a_1,\dots,\a_n$. To define root vectors for all positive roots consider a reduced decomposition of the longest element of the Weyl group $w_0=s_{i_1}\dots s_{i_N}$. Then all positive roots occur exactly once in $$\b_1=\a_{i_1}, \b_2=s_{i_2}(\a_{i_2}),\dots, \b_N=s_{i_1}\dots s_{i_{N-1}}(\a_{i_N}).$$
Then we define $$E_{\b_k}=T_{i_1}\dots T_{i_{k-1}}(E_{i_k}), \ \ F_{\b_k}=T_{i_1}\dots T_{i_{k-1}}(F_{i_k}).$$
where the $T_i$'s are determined by the braid action of $B(\gg)$ on $\UHqg$ \cite[Theorem 8.1.2]{CP:BOOK}: 
\begin{align*}
    T_i(E_j)&=\sum_{k=0}^{-a_{ij}}(-1)^{k-a_{ij}}\frac{q^{-k}}{[-a_{ij}-k]_{\qi}![k]_{\qi}!}E_i^{-a_{ij}-k}E_jE_i^k,\\
     T_i(F_j)&=\sum_{k=0}^{-a_{ij}}(-1)^{k-a_{ij}}\frac{q^{k}}{[-a_{ij}-k]_{\qi}![k]_{\qi}!}F_i^kF_jF_i^{-a_{ij}-k}
\end{align*}
for $i\neq j$, and $T_i(E_i)=-F_iK_i, T_i(F_i)=-K_i^{-1}E_i$.
\medskip

Note that $T_i(E_j)$ has $L_R$-degree $\a_j-a_{ij}\a_i=s_i(\a_j)$. It follows that $E_{\b_k}$ has degree $s_{i_1}\dots s_{i_{k-1}}(\a_{i_k})=\b_{k}$. Thus, $d(E_{\b})=\b$ and similarly $d(F_{\b})=-\b$ for all $\b\in \Dep$.
\medskip

Now for every $I:\Dep\to \N_{\geq 0}$ set $$E^I:=E_{\b_N}^{I(\b_N)}\cdots E_{\b_1}^{I(\b_1)}, \ \ F^I:=F_{\b_N}^{I(\b_N)}\cdots F_{\b_1}^{I(\b_1)}.$$
If for every $I$ we denote $$d(I)=\sum_{\b\in\Dep}I(\b)\b\in L_R,$$ then $d(E^I)=d(I)$ and $d(F^I)=-d(I)$.

\begin{proposition}(\cite{CP:BOOK})\label{lemma: PBW belongs to V+}
For each positive root $\b$, $E_{\b}\in V_+$ and $\{E^I \ | \ I:\Dep\to \N_{\geq 0}\}$ is a basis of $V^+$. Similarly $F_{\b}\in V_-$ and $\{F^I \ | \ I:\Dep\to \N_{\geq 0}\}$ is a basis of $V^-.$ 
\end{proposition}

\subsection{Unrolled restricted and small quantum groups} 


The {\em restricted} unrolled quantum group is the quotient $\UHr$ of $\UHqg$ by the two-sided ideal generated by $E_{\b}^{r'}$ and $F_{\b}^{r'}$ for each $\b\in\Dep$. Here $r'=r$ if $r$ is odd and $r'=r/2$ if $r$ is even. Note that this ideal is a Hopf ideal, this is because $E_i\ot K_i$ and $1\ot E_i$ $\qi^2$-commute for all $i$ (in the simply-laced case). The subalgebra of $\UHr$ generated by $E_i,F_i,K_i^{\pm 1}$ is denoted $\Ur$ and simply called the restricted quantum group. The {\em small quantum group }$\uqg$ is the quotient of $\Ur$ by setting $K_i^r=1$ for all $i=1,\dots, n$. The {\em small Borel} is the subalgebra $\uqb$ of $\uqg$ generated by $K_i^{\pm 1}, E_i$ for $i=1,\dots,n$.
\medskip

\begin{proposition}
    The set $$\{K_{\c}E^IF^J \ | \ \c\in L_R \text{ and } I,J:\Dep\to\{0,\dots,r'-1\}\}$$ forms a basis of $\uqg$.
\end{proposition}

\subsection{Weight modules}
\def\laa{[\la]}
\def\muu{[\mu]}

A {\em weight module} is a $\UHr$-module for which $\Uh$ acts diagonalizably and $K_i$ acts as $\qi^{H_i}$ for every $i$. Let $V$ be a weight module. For each $\la\in\hh^*$ let $$V(\la)=\{v\in V \ | \ Hv=\la(H)v, H\in \hh\}.$$ Then $E_i(V(\la))\sb V(\la+\a_i)$ and $F_i(V(\la))\sb V(\la-\a_i)$. Thus, for any $\la\in\hh^*$, $\oplus_{\c\in L_R}V(\la+\c)$ is a submodule of $V$. Let $G=\hh^*/L_R$ and let $\CC_{\laa}$ be the category of weight modules whose weights are in the class of $\laa:=\la+L_R\in G$. Then the whole category $\CC$ of weight modules splits as a direct sum $$\CC=\bigoplus_{\laa\in G}\CC_{\laa}.$$
Moreover, one has $\CC_{\laa}\ot\CC_{\muu}\sb \CC_{\laa+\muu}$ for each $\laa,\muu\in G$.
The neutral component consists of modules with weights in $L_R$. Modules in $\CC_{[0]}$ are modules over the small quantum group $\uqg$ (where $K_i^r=1$ for all $i$) with a choice of logarithm for the weights (the weights of $\uqg$ are only defined mod $r$). Thus, there is a forgetful functor $F:\CC_{[0]}\to \Rep(\uqg)$. 

\subsection{Verma modules} For each $\la\in \hh^*$, let $V_{\la}$ be the (finite dimensional) $\ov{U}_q(\gg)$-module generated by a highest weight vector $v_0$ of highest weight $\la$. In other words $V_{\la}=\ov{U}_q(\gg)\ot_{U^+}\C$ where $\C$ is a $U^+$-module via the map $\xi:U^+\to\C$ defined by $\xi(E_i)=0$ and $\xi(K_i)=q^{\la(H_i)}$ for every $i$. Here $U^+$ denotes the subalgebra of $\ov{U}_q(\gg)$ generated by the $K_i^{\pm 1}, E_i$ for $i=1,\dots,n$. Then $V_{\la}$ is a weight $\UHr$-module in an obvious way. It has a basis $v_A:=F^Av_0$ for $A:\Dep\to \{0,\dots, r'-1\}$, $v_A$ has weight $\la-d(A)$.
 \medskip

We call $\la$ {\em typical} if $q^{2( \la+\rho,\b)-m(\b,\b)}\neq 1$ for all $\b\in\Dep$ and $m=1,\dots,r'-1$.

\begin{proposition}(\cite{GP:trace-projective-quantum-groups, Rupert:unrolled})
    For typical $\la$, $V_{\la}$ is simple and projective in $\CC$.
\end{proposition}

\def\rev{\overrightarrow{\text{ev}}}

\def\rcoev{\overrightarrow{\text{coev}}}
\def\lev{\overleftarrow{\text{ev}}}

\def\lcoev{\overleftarrow{\text{coev}}}

\def\coev{\text{coev}}
\def\ev{\text{ev}}
\def\XX{\mathbb{X}}

\subsection{Pivotal structure} If $V$ is a weight $\UHr$-module, then so does $V^*$ with the action $x\cdot v'(v)=v'(S(x)v)$ for all $v\in V, v'\in V^*, x\in\UHr$. The weight spaces are $V^*(\mu)=V(-\mu)^*$, thus $V^*\in \CC_{-\laa}$ if $V\in\CC_{\laa}$. Left evaluations and coevaluations are the usual ones of vector spaces, but right evaluations and coevaluations are given by $$\rev_V(v\ot v')=v'(gv), \ \ \rcoev_V(1)=\sum v_i\ot g^{-1}v_i'$$
where $g=K_{2\rho}^{1-r'}$. Here $(v_i)$ is a basis of $V$ and $(v_i')$ is the dual basis. It follows that there is an isomorphism $j_V:V^{**}\to V$ such that $$j_V^{-1}(v)(v')=v'(gv).$$
For $V_{\la}$ one has 
\begin{align}
    \label{eq: pivotal on Verma}
    j_{V_{\la}}(v''_A)=q^{(r'-1)(\la-d(A),2\rho)}v_A
\end{align}

where $v''_A$ is the basis of $V_{\la}^{**}$ determined by $v''_A(v'_B)=\d_{A,B}$.

\def\HH{\mathcal{H}}\def\KK{\mathcal{K}}
\def\TT{\ov{\Theta}}

\def\qa{q}

\subsection{Braiding on weight modules}
\label{subs: braiding on weight} 
The category of weight $\UHr$-modules is braided with $c_{V,W}:V\ot W\to W\ot V$ defined by $$c_{V,W}(v\ot w)= \tau_{V,W}(\HH\TT(v\ot w))$$
where $$\TT=\sum_{I:\Dep\to\{0,\dots,r'-1\}} c_{I} \cdot E^I\ot F^I$$
with $c_I=\prod_{\a\in\Dep}\frac{(\qa-\qa^{-1})^{I(\a)}}{[I(\a)]_q!}\qa^{\frac{I(\a)(I(\a)-1)}{2}}$, and
$\HH$ denotes the operator $V\ot W\to V\ot W$ defined by $$\HH(v\ot w)=q^{(\la,\mu)}v\ot w$$ if $v\in V(\la)$ and $w\in W(\mu)$.
\medskip

\def\bk{\boldsymbol{k}}
\def\bl{\boldsymbol{l}}
\def\bc{\boldsymbol{c}}
\def\bd{\boldsymbol{d}}
\def\bx{\boldsymbol{x}}
\def\by{\boldsymbol{y}}
\begin{lemma}
\label{lemma: unrolled R-matrix restrict to ROSSO matrix in small quantum}
If $V,W\in \CC_{[0]}$ the operator $\HH$ acts on $V\ot W$ by multiplication by $$\ov{\KK}=\frac{1}{r^n}\sum_{\a,\b\in Q_r}q^{-(\a,\b)}K_{\a}\ot K_{\b}$$
where $Q_r=\{c_1\a_1+\dots +c_n\a_n \ | \ 0\leq c_i\leq r-1\}\sb L_R$ and $(\a_i,\a_j)=a_{ij}$.
\end{lemma}
\begin{proof}
   Let $v\in V(\la)$ and $w\in W(\mu)$ where $\la=\sum_i c_i\a_i\in L_R$ and $\mu=\sum_i d_i\a_i\in L_R$. Then
   \begin{align*}
       \ov{\KK}(v\ot w)&=\frac{1}{r^n}\sum_{\a,\b\in Q_r}q^{-(\a,\b)+(\a,\la)+(\b,\mu)}(v\ot w)\\
       &=\frac{1}{r^n}\sum_{\bk,\bl \in (\Z/r)^n}q^{-\bk^tA\bl+\bk^tA\bc+\bd^tA\bl}(v\ot w)\\
       &=\frac{1}{r^n}\sum_{\bk}q^{\bk^tA\bc}\sum_{\bl}q^{(-\bk^t+\bd^t)A\bl}(v\ot w)\\
        &=\frac{1}{r^n}\sum_{\bk}q^{\bk^tA\bc}\sum_{\bx}q^{(-\bk^t+\bd^t)\bx}(v\ot w)\\
        &=\frac{1}{r^n}\sum_{\bk}q^{\bk^tA\bc} r^n\delta_{-\bk+\bd,0}(v\ot w)\\
        &=q^{\bd^tA\bc}(v\ot w)=q^{(\la,\mu)}(v\ot w)=\HH(v\ot w).
   \end{align*}
   Note that in the fourth equality we used our assumption that $A$ is invertible mod $r$ and in the fifth equality we used that $\sum_{\bx\in (\Z/r)^n}q^{\by^t\cdot\bx}=r^n\d_{\by,\boldsymbol{0}}$ since $q$ is a primitive $r$-th root of unity.
\end{proof}

\begin{corollary}
\label{cor: unrolled at deg 0 is small quantum}
    The forgetful functor $F:\CC_{[0]}\to \Rep(\uqg)$ is braided.
\end{corollary}

Here we consider $\Rep(\uqg)$ with the braiding given by $c_{V,W}(v\ot w)=\tau_{V,W}(\ov{\KK}\cdot\TT(v\ot w)).$


\subsection{Ribbon structure} The category $\CC$ is ribbon \cite{GP:trace-projective-quantum-groups} in a way compatible with the above braiding and pivotal structure. The twist on a simple $V_{\la}$ is given by 
\begin{align}
\label{eq: ribbon twist on Verma}
    \theta_{V_{\la}}=q^{(\la,\la-(r'-1)2\rho)}\id_{V_{\la}}.
\end{align}


\medskip

\subsection{Lemmas on PBW} We now prove some simple lemmas concerning the algebra structure on $\Uqg$.

\begin{lemma}
\label{lemma: multip 2 product of EI's}
    For any $I,J$ one can write $F^IF^J=\sum c_{A}(q)F^A$ for some $c_{A}(q)\in\Q(q)$ and $d(A)=d(I)+d(J)$.
\end{lemma}
\begin{proof}
    This is obvious. This lemma only says that $V_-$ is graded by the root lattice and multiplication preserves the grading (and that $F^I$ has degree $-d(I)$).
\end{proof}

\begin{lemma}
\label{lemma: multip 1 non-simple root is product}
    If $I$ is not a simple root then $F^I=\sum c_{A,B}(q)F^AF^B$ where $c_{A,B}(q)\in \Q(q)$ and $d(A)+d(B)=d(I)$ with $0< d(A), d(B)< d(I)$. 
\end{lemma}
\begin{proof}
Clearly, we only need to show this when $F^I=F_{\b_k}$ and $\b_k$ is not a simple root. Note first that from the above formula of $T_i(F_j)$ one sees that $a_{ij}=0$ implies $T_i(F_j)=F_j$ and $a_{ij}<0$ implies $T_i(F_j)$ (whose degree is $\a_j-a_{ij}\a_i$) is a sum of products of $F_l$'s of smaller $L_R$-degree. Indeed, note that if $0<k\leq -a_{ij}$ then $0<k\a_i,(-a_{ij}-k)\a_i+\a_j<\a_j-a_{ij}\a_i$ and if $k=0$, $0<\a_j,-a_{ij}\a_i<\a_j-a_{ij}\a_i$.
Now, if $\b_{k}$ is a non-simple positive root, then there is an index $i_l$ (in the decomposition $w_0=s_{i_1}\dots s_{i_N}$) with $1\leq l\leq k-1$ such that $a_{i_l,i_{l+1}}<0$. Since the $T_i$'s are algebra automorphisms, it follows from the previous discussion that $F_{\b_k}=T_{i_1}\dots T_{i_{k-1}}(F_{i_k})$ is a sum of products of $F_i$'s of smaller $L_R$-degree. 
\medskip
\end{proof}

Denote $[q^nK_i]=\frac{q^nK_i-q^{-n}K_i^{-1}}{q-q^{-1}}$. Note that $K_iF^A=q^{-(d(A),\a_i)}F^AK_i$ so that $[q^nK_i]F^A=F^A[q^{n-(d(A),\a_i)}K_i]$.

\begin{lemma}
\label{lemma: commutation for simple root}
 For every $i$ and $A:\Dep\to \{0,\dots,r'-1\}$ we have   $$E_iF^A=F^AE_i+\sum_Bc_B(q) F^B[q^{n_B}K_i]$$
 for some $c_B(q)\in \Q(q)$ and $n_B\in\Z$.
\end{lemma}
\begin{proof}
    If $A$ is a simple root this is obvious. If not, we can use Lemma \ref{lemma: multip 1 non-simple root is product} to write $F^A=\sum c_{B,C}(q)F^BF^C$ with $d(B),d(C)<d(A)$. By induction
    \begin{align*}
        E_iF^A&=\sum c(q)E_iF^BF^C\\
        &=\sum c(q) (F^BE_i+\sum_Dd(q)F^D[q^{n_D}K_i])F^C\\
        &=\sum c(q)F^BE_iF^C+\sum f(q)F^DF^C[q^{n_D-(d(C),\a_i)}K_i]\\
        &=\sum c(q)F^B(F^CE_i+\sum e(q)F^I[q^{n_I}K_i])+\sum f(q)F^DF^C[q^{n_D-(d(C),\a_i)}K_i]\\
        &=F^AE_i+\sum g(q)F^BF^I[q^{n_I}K_i]+\sum f(q)F^DF^C[q^{n_D-(d(C),\a_i)}K_i].
    \end{align*}
    The proof is finished by writing each $F^BF^I, F^DF^C$ as a $\Q(q)$-linear combination of $F^J$'s (Lemma \ref{lemma: multip 2 product of EI's}).
\end{proof}

\subsection{Lemmas on Verma modules} The Verma $V_{\la}$ has basis the vectors $v_A:=F^Av_0$ where $A:\Dep\to \{0,\dots,r'-1\}$. We will consider $V^*_{\la}$ with the dual basis $v'_A$, that is, $\lb v'_A,v_B\rb=\d_{A,B}$ and $X=V_{\la}\ot V_{\la}^*$ with the tensor product basis. We refer to these bases (and tensor products of these) as standard bases. We consider $q^{\la(H_1)},\dots, q^{\la(H_n)}$ as independent variables. For simplicity, we denote $\la_i=\la(H_i)$.

\begin{lemma}
    In the standard basis of $V_{\la}$, $E_i$ acts by a matrix with coefficients of the form $a(q)q^{\la_i}+b(q)q^{-\la_i}$ with $a(q),b(q)\in \Q(q)$.  In the dual basis of $V_{\la}^*$, $E_i$ acts by a matrix with coefficients of the form $c(q)+d(q)q^{-2\la_i}$, $c(q),d(q)\in \Q(q)$.
\end{lemma}
\begin{proof}
     Using $E_iv_0=0$, $K_iv_0=q^{\la_i}v_0$ and Lemma \ref{lemma: commutation for simple root} we find: 
    \begin{align*}
       E_iv_A&=E_iF^Av_0=F^AE_iv_0+\sum_Bc_B(q) F^B[q^{n_B}K_i]v_0\\
       &=\sum_B c_B(q)[n_B+\la_i]_qv_B=\sum_B c'_B(q)(q^{n_B+\la_i}-q^{-n_B-\la_i})v_B  \\
       &=\sum_B (a_B(q)q^{\la_i}+b_B(q)q^{-\la_i})v_B 
    \end{align*}
    This shows the first part of the lemma. Now, the matrix representing $E_i$ acting on $V_{\la}^*$ in the dual basis is the transpose of the matrix of $S(E_i)=-E_iK_i^{-1}$ acting on $V_{\la}$. But $K_i^{-1}$ is represented by a diagonal matrix with coefficients of the form $q^{m_B-\la_i}, m_B\in\Z$ hence multiplying this with the previous matrix of $E_i$, we get a matrix with coefficients of the desired form.
\end{proof}

In the following lemmas we consider $X=V_{\la}\ot V_{\la}^*$ as a $\uqg$-module.

\begin{lemma}
  In the standard basis of $X$,  $E_i$ acts by matrices whose coefficients have the form $c(q)+d(q)q^{-2\la_i}$ with $c(q),d(q)\in\Q(q)$. In the same basis of $X$, $F_i$ acts by a matrix with $\Q(q)$-coefficients.
\end{lemma}
\begin{proof}

    Indeed, $K_iv'_B=q^{-\la_i+m_B}v'_B$ for some $m_B\in\Z$. Hence the matrix representing the action of $E_i\ot K_i$ on $X$ has coefficients of the form $$(a(q)q^{\la_i}+b(q)q^{-\la_i})q^{-\la_i+m_B}=c(q)+d(q)q^{-2\la_i}$$
    for some $c(q),d(q)\in \Q(q)$. Now, by the previous lemma, $E_i$ acts on $V_{\la}^*$ by a matrix of the same form, hence so does the matrix of $1\ot E_i$ on $X$. Thus, $E_i\ot K_i+1\ot E_i$ acts by such a matrix, and this is exactly the matrix of the action of $E_i$ on $X$. The second assertion follows from $\De(F_i)=K_i^{-1}\ot F_i+F_i\ot 1$ and $S(F_i)=-K_iF_i$: the $q^{-\la_i}$'s of the action of $K_i^{-1}$ on $V_{\la}$ cancel with the $q^{\la_i}$'s of the action of $-K_iF_i$ on $V_{\la}$, hence only $\Q(q)$-coefficients appear.
\end{proof}

\begin{proposition}
\label{lemma: ESSENTIAL LEMMA}
    Let $z\in\uqg$ and consider its action on the module $X=V_{\la}\ot V_{\la}^*$. Then, in the standard basis of $X$, each coefficient of the matrix representing $z$ is a polynomial in $\Q(q)[q^{-2\la_1 },\dots, q^{-2\la_n}]$ where the power of each $q^{-2\la_i}$ appearing in some monomial is bounded above by $(r'-1)k_i$, where $k_i$ is the coefficient of $\a_i$ in $2\rho$.
\end{proposition}
\begin{proof}
    It suffices to prove this when $z$ is a PBW basis element. Since any $K_{\c}$ and $F^I$ acts on $X$ with $\Q(q)$-coefficients, we can further suppose $z=E^I$ for some $I:\Dep\to \{0,\dots, r'-1\}$. If $d(I)=l_1\a_1+\dots+l_n\a_n$, then $E^I$ is a $\Q(q)$-linear combination of products $E_{i_1}E_{i_2}\dots E_{i_m}$ where there are $l_i$ repetitions of $i$ among the $i_1,i_2,\dots,i_m$. By the previous lemma, the matrix representing $E_{i_1}\dots E_{i_m}$ has coefficients of the form $$(c_1(q)+d_1(q)q^{-2\la_{i_1}})\dots (c_m(q)+d_m(q)q^{-2\la_{i_m}}).$$
    In this product, there are at most $l_i$ terms with a $q^{-2\la_i}$ for every $i$. Thus, in every monomial appearing when expanding this product, the power of $q^{-2\la_i}$ is bounded above by $l_i$. Since $l_i\leq (r'-1)k_i$ for every $i,$ the proposition follows.
\end{proof}

\def\Urqg{\overline{U}_q(\gg)}

\def\UHrqg{\overline{U}^H_q(\gg)}

\section{Link invariants from $\UHrqg$ and genus bounds}  
\label{section: link invariants}

\def\CCC{\mathcal{C}}
\def\ttt{\boldsymbol{t}}\def\Proj{\text{Proj}}

In this section, we define the $\ADOgr'$ invariant of links $L\sb S^3$ and we prove a Seifert surface formula for these. Since this formula involves universal $\uqg$-invariants we briefly recall these in Subsections \ref{subs: universal invariants} and \ref{subs: RT invariants}. The proof of Theorem \ref{theorem: ADO genus} will follow easily from this formula and Proposition \ref{lemma: ESSENTIAL LEMMA}, see Subsection \ref{subs: Proof of THM 1}. We give some computations for $\ADOgr'$ of knots to illustrate our theorem.\\

\def\ev{\text{ev}}
\def\coev{\text{coev}}

In what follows we denote by $\coev$ the usual right coevaluation of vector spaces and by $\tau_{V,W}:V\ot W\to W\ot V$ be the usual vector space transposition $\tau(x\ot y)=y\ot x$.

\subsection{Universal $\uqg$-invariants}
\label{subs: universal invariants}
Let $T$ be a framed, oriented tangle in $\R^2\t[0,1]$
with ordered components, say $T_1,\dots, T_l$. Let $D$ be a planar diagram of $T$ which we suppose comes with the blackboard framing. We can always isotope $D$ so that it is formed from upward oriented crossings (positive or negative), and caps and cups. To every crossing we associate the $R$-matrix of $\uqg$ or its inverse, and to each cap or cup we associate group-likes as follows:

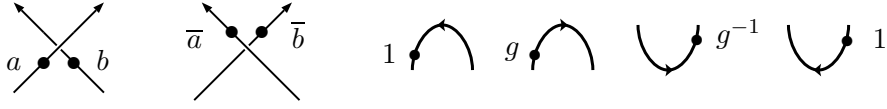
\begin{figure}[H]
\begin{pspicture}(0,-0.6424193)(11.64,0.6424193)
\psbezier[linecolor=black, linewidth=0.04](7.8,-0.21484207)(7.8,0.08515793)(7.6,0.38515794)(7.4,0.38515792846679686)(7.2,0.38515794)(7.0,0.08515793)(7.0,-0.21484207)
\psline[linecolor=black, linewidth=0.04, arrowsize=0.09cm 2.0,arrowlength=0.6,arrowinset=0.2]{->}(7.38,0.38515794)(7.4666667,0.38515794)
\psbezier[linecolor=black, linewidth=0.04](9.2,0.38515794)(9.2,0.08515793)(9.0,-0.21484207)(8.8,-0.21484207153320312)(8.6,-0.21484207)(8.4,0.08515793)(8.4,0.38515794)
\psline[linecolor=black, linewidth=0.04, arrowsize=0.09cm 2.0,arrowlength=0.6,arrowinset=0.2]{->}(8.78,-0.21484207)(8.866667,-0.21484207)
\psdots[linecolor=black, dotsize=0.14](9.18,0.18515792)
\psdots[linecolor=black, dotsize=0.14](7.03,-0.014842072)
\rput[bl](9.44,0.06515793){$g^{-1}$}
\rput[bl](6.63,-0.09484207){$g$}
\psbezier[linecolor=black, linewidth=0.04](10.4,0.38515794)(10.4,0.08515793)(10.6,-0.21484207)(10.8,-0.21484207153320312)(11.0,-0.21484207)(11.2,0.08515793)(11.2,0.38515794)
\psline[linecolor=black, linewidth=0.04, arrowsize=0.09cm 2.0,arrowlength=0.6,arrowinset=0.2]{->}(10.82,-0.21484207)(10.733334,-0.21484207)
\psbezier[linecolor=black, linewidth=0.04](5.4,-0.21484207)(5.4,0.08515793)(5.6,0.38515794)(5.8,0.38515792846679686)(6.0,0.38515794)(6.2,0.08515793)(6.2,-0.21484207)
\psline[linecolor=black, linewidth=0.04, arrowsize=0.09cm 2.0,arrowlength=0.6,arrowinset=0.2]{->}(5.82,0.38515794)(5.733333,0.38515794)
\psdots[linecolor=black, dotsize=0.14](5.43,-0.014842072)
\psdots[linecolor=black, dotsize=0.14](11.18,0.17515793)
\rput[bl](5.0,-0.11484207){$1$}
\rput[bl](11.52,0.08515793){$1$}
\psline[linecolor=black, linewidth=0.026](2.5,-0.61484206)(3.4,0.28515792)
\psline[linecolor=white, linewidth=0.026, doubleline=true, doublesep=0.026, doublecolor=black](3.0,0.28515792)(3.9,-0.61484206)
\psline[linecolor=black, linewidth=0.026](1.3,-0.5148421)(0.5,0.28515792)
\psline[linecolor=white, linewidth=0.026, doubleline=true, doublesep=0.026, doublecolor=black](0.9,0.28515792)(0.1,-0.5148421)
\psline[linecolor=black, linewidth=0.026, arrowsize=0.05291667cm 2.0,arrowlength=1.4,arrowinset=0.0]{<-}(0.1,0.68515795)(0.5,0.28515792)
\psline[linecolor=black, linewidth=0.026, arrowsize=0.05291667cm 2.0,arrowlength=1.4,arrowinset=0.0]{<-}(1.3,0.68515795)(0.9,0.28515792)
\psline[linecolor=black, linewidth=0.026, arrowsize=0.05291667cm 2.0,arrowlength=1.4,arrowinset=0.0]{<-}(2.6,0.68515795)(3.0,0.28515792)
\psline[linecolor=black, linewidth=0.026, arrowsize=0.05291667cm 2.0,arrowlength=1.4,arrowinset=0.0]{<-}(3.8,0.68515795)(3.4,0.28515792)
\psdots[linecolor=black, dotsize=0.16](0.5,-0.11484207)
\psdots[linecolor=black, dotsize=0.16](0.9,-0.11484207)
\psdots[linecolor=black, dotsize=0.16](3.0,0.28515792)
\psdots[linecolor=black, dotsize=0.16](3.4,0.28515792)
\rput[bl](0.0,-0.21484207){$a$}
\rput[bl](1.2,-0.21484207){$b$}
\rput[bl](2.4,0.08515793){$\overline{a}$}
\rput[bl](3.8,0.08515793){$\overline{b}$}
\end{pspicture}
\end{figure}

Here we denote $R=\sum a\ot b$, $R^{-1}=\sum \overline{a}\ot \overline{b}$ and $g=K_{2\rho}^{1-r'}$. Then, as we follow the orientation of a component $T_i$, we multiply the elements encountered from right to left. Doing this for all components gives an element $$Z_T\in \uqg^{\ot l}$$
where the $i$-th slot corresponds to the product along the $i$-th component. For more details, see \cite{Ohtsuki:BOOK, Habiro:bottom-tangles}.

\begin{example}
\label{example: ZT of tangle}
For the following tangle $T$ (ordered left to right)

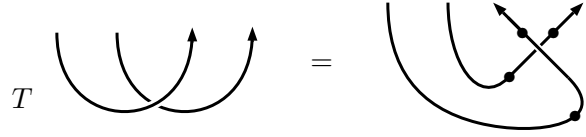
\begin{figure}[H]
\label{figure: ZT of tangle}
   \psscalebox{1.0 1.0} 
{
\begin{pspicture}(0,-0.86018735)(7.609526,0.86018735)
\psline[linecolor=black, linewidth=0.04, arrowsize=0.05291667cm 2.0,arrowlength=1.4,arrowinset=0.0]{->}(6.6,-0.13981262)(7.6,0.86018735)
\psline[linecolor=white, linewidth=0.026, doubleline=true, doublesep=0.04, doublecolor=black](6.6,0.66018736)(7.4,-0.13981262)
\psbezier[linecolor=black, linewidth=0.04](1.4,0.46018738)(1.4,-0.9398126)(3.2,-0.9398126)(3.2,0.4601873779296875)
\psbezier[linecolor=white, linewidth=0.04, doubleline=true, doublesep=0.04, doublecolor=black](0.6,0.46018738)(0.6,-0.9398126)(2.4,-0.9398126)(2.4,0.4601873779296875)
\psbezier[linecolor=black, linewidth=0.04](5.0,0.86018735)(5.0,-0.006479289)(5.4,-0.5398126)(6.2,-0.7398126220703125)(7.0,-0.9398126)(8.0,-0.7398126)(7.4,-0.13981262)
\psbezier[linecolor=black, linewidth=0.04](5.8,0.86018735)(5.8,0.060187377)(6.2,-0.5398126)(6.6,-0.1398126220703125)
\psline[linecolor=black, linewidth=0.04, arrowsize=0.05291667cm 2.0,arrowlength=1.4,arrowinset=0.0]{->}(2.4,0.33018738)(2.4,0.53018737)
\psline[linecolor=black, linewidth=0.04, arrowsize=0.05291667cm 2.0,arrowlength=1.4,arrowinset=0.0]{->}(3.2,0.34018737)(3.2,0.54018736)
\psline[linecolor=black, linewidth=0.04, arrowsize=0.05291667cm 2.0,arrowlength=1.4,arrowinset=0.0]{->}(6.8,0.46018738)(6.4,0.86018735)
\psdots[linecolor=black, dotsize=0.14](6.79,0.46018738)
\psdots[linecolor=black, dotsize=0.14](7.2,0.46018738)
\psdots[linecolor=black, dotsize=0.14](6.62,-0.12981263)
\psdots[linecolor=black, dotsize=0.14](7.49,-0.64981264)
\rput[bl](3.97,-0.029812623){=}
\rput[bl](0.0,-0.5398126){$T$}
\end{pspicture}
}

\end{figure}

\noindent  one has $$Z_T=\sum\overline{a}g^{-1}\ot \overline{b}g^{-1}\in \uqg^{\ot 2}.$$
Note that on the right we isotoped $T$ so that both strands at the crossing are oriented upwards.
\end{example}

\def\FC{F}

\subsection{Reshetikhin-Turaev invariants}\label{subs: RT invariants} We assume the reader is familiar with the usual Reshetikhin-Turaev construction for framed oriented tangles colored with objects in a ribbon category, which in our case is either the category $\CC$ of weight modules over $\UHrqg$ or $\Rep(\uqg)$ \cite{RT1}. We will only consider tangles $T$ where all components are colored with the same module $V$, so we denote the corresponding operator invariant by $\FC(T,V)$. Now let $X$ be a $\uqg$-module. Then, it is well-known that the universal invariant $Z_T$ determines the operator invariant $\FC(T,X)$ essentially by left multiplication on the corresponding module, see \cite{Habiro:bottom-tangles, Ohtsuki:BOOK}. For instance, if $T$ is the 2-component tangle of Example \ref{example: ZT of tangle}, then $\FC(T,X):\C\to (X^*\ot X^*\ot X\ot X)$
satisfies
\begin{align*}
    \FC(T,X)(1)&=(\id_{X^*}\ot\tau_{X,X^*}\ot \id_X)\circ L_{Z_T^{ev}}\circ (\coev_X\ot\coev_X)(1)\\
    &=\sum x_i'\ot x_j'\ot \overline{a}g^{-1}x_i\ot \overline{b}g^{-1}x_j
\end{align*}
where $(x_i)$ is a basis of $X$ and $(x'_i)$ the dual basis. Here $L_{Z_T^{ev}}:(X^*\ot X)^{\ot 2}\to (X^*\ot X)^{\ot 2}$ denotes left multiplication by $$Z_T^{ev}=\sum 1_{\uqg}\ot \overline{a}g^{-1}\ot 1_{\uqg}\ot\overline{b}g^{-1}.$$

\subsection{ADO invariants}\label{subs: ADO invariants} Let $L$ be a framed, oriented link presented as the closure of a $(1,1)$-tangle $L_o$. Let $\la\in\hh^*$ and suppose every component of $L_o$ is colored with $V_{\la}$. Then, the Reshetikhin-Turaev construction gives an operator $F(L_o,V_{\la}):V_{\la}\to V_{\la}$. For typical $\la$, $V_{\la}$ is simple so this operator is multiplication by a scalar $\lb L_o\rb$. We define $$\ADOgr(L,\la):=\lb L_o\rb$$
and we claim this is a link invariant. Note that the $(1,1)$-tangle $L_o$ is uniquely defined only when $L$ is a knot. For general links, the fact that the above is still a link invariant follows from the following argument: the category $\Proj(\CC)$ admits a two-sided modified trace $\ttt$ \cite{Rupert:unrolled}, hence $\ttt(F(L_o))$ is a link invariant by \cite{GPV:traces-pivotal-cats}. Since $V_{\la}$ is simple $\ttt(F(L_o,V_{\la}))=\lb L_o\rb \ttt(\id_{V_{\la}})$ and since all components of $L$ have the same color, it follows that $\lb L_o\rb$ is a link invariant as well.\\

Using the formula for the twist $\tt_{V_{\la}}$ of (\ref{eq: ribbon twist on Verma}), we can define an invariant of unframed, oriented links by setting 
\begin{align}
\label{eq: unframed ADO}
    \ADOgr'(L,\la):=q^{-w(L)(\la,\la-(r'-1)2\rho)}\ADOgr(L,\la)
\end{align}
where $w(L)$ denotes the writhe of $L$. 

\subsection{Seifert surface formula}\label{subs: Seifert formula} Let $L$ be an unframed, oriented link with $s$ components and let $S$ be a Seifert surface for $L$. Then, after an isotopy, $S$ can be obtained by thickening a framed tangle $T$ with $2g+s-1$ components and all endpoints on $\R^2\t\{1\}$ and then attaching a disk on top, see Figure \ref{figure:Seifert}. We will consider $L$ as a framed link with the framing coming from $S$, thus $w(L)=0$ and $\ADOgr(L)=\ADOgr'(L)$. Let $L_o$ be the framed $(1,1)$-tangle obtained by opening $L$ at $p$ as indicated in the same figure.

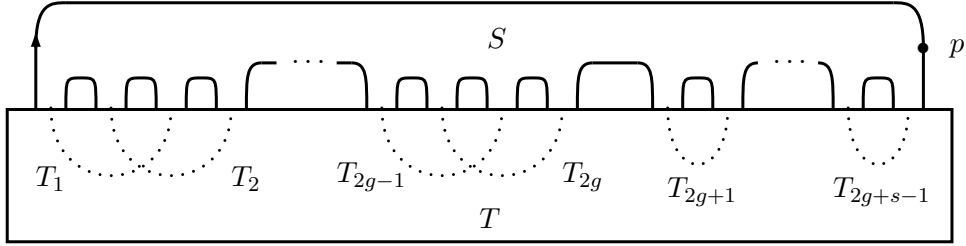
\begin{figure}[H]
   \centering
   \psscalebox{1.0 1.0} 
{
\begin{pspicture}(0,-1.61)(12.72,1.61)
\psframe[linecolor=black, linewidth=0.04, dimen=outer](12.6,0.19)(0.0,-1.61)
\psbezier[linecolor=black, linewidth=0.04](0.8,0.19)(0.8,0.59)(0.8,0.59)(1.0,0.59)(1.2,0.59)(1.2,0.59)(1.2,0.19)
\psbezier[linecolor=black, linewidth=0.04](1.6,0.19)(1.6,0.59)(1.6,0.59)(1.8,0.59)(2.0,0.59)(2.0,0.59)(2.0,0.19)
\psline[linecolor=black, linewidth=0.04, arrowsize=0.05291667cm 2.0,arrowlength=1.4,arrowinset=0.0]{->}(0.4,0.19)(0.4,1.19)
\psbezier[linecolor=black, linewidth=0.04](2.4,0.19)(2.4,0.59)(2.4,0.59)(2.6,0.59)(2.8,0.59)(2.8,0.59)(2.8,0.19)
\psbezier[linecolor=black, linewidth=0.04](3.2,0.19)(3.2,0.59)(3.2,0.79)(3.4,0.79)
\psbezier[linecolor=black, linewidth=0.04](5.2,0.19)(5.2,0.59)(5.2,0.59)(5.4,0.59)(5.6,0.59)(5.6,0.59)(5.6,0.19)
\psbezier[linecolor=black, linewidth=0.04](6.0,0.19)(6.0,0.59)(6.0,0.59)(6.2,0.59)(6.4,0.59)(6.4,0.59)(6.4,0.19)
\psbezier[linecolor=black, linewidth=0.04](6.8,0.19)(6.8,0.59)(6.8,0.59)(7.0,0.59)(7.2,0.59)(7.2,0.59)(7.2,0.19)
\psbezier[linecolor=black, linewidth=0.04](4.8,0.19)(4.8,0.59)(4.8,0.79)(4.6,0.79)
\rput[bl](3.8,0.79){$\dots$}
\psbezier[linecolor=black, linewidth=0.04, linestyle=dotted, dotsep=0.10583334cm](0.6,0.19)(0.6,-1.01)(2.2,-1.01)(2.2,0.19)
\psbezier[linecolor=black, linewidth=0.04, linestyle=dotted, dotsep=0.10583334cm](1.4,0.19)(1.4,-1.01)(3.0,-1.01)(3.0,0.19)
\psbezier[linecolor=black, linewidth=0.04, linestyle=dotted, dotsep=0.10583334cm](5.0,0.19)(5.0,-1.01)(6.6,-1.01)(6.6,0.19)
\psbezier[linecolor=black, linewidth=0.04, linestyle=dotted, dotsep=0.10583334cm](5.8,0.19)(5.8,-1.01)(7.4,-1.01)(7.4,0.19)
\rput[bl](6.3,-1.41){$T$}
\psline[linecolor=black, linewidth=0.04](12.2,0.19)(12.2,0.99)
\psbezier[linecolor=black, linewidth=0.04](0.4,0.99)(0.4,1.59)(0.6,1.59)(0.8,1.59)
\psbezier[linecolor=black, linewidth=0.04](12.2,0.99)(12.2,1.59)(12.0,1.59)(11.8,1.59)
\psline[linecolor=black, linewidth=0.04](0.8,1.59)(11.8,1.59)
\psline[linecolor=black, linewidth=0.04](3.4,0.79)(3.6,0.79)
\psline[linecolor=black, linewidth=0.04](4.4,0.79)(4.6,0.79)
\psdots[linecolor=black, dotsize=0.14](12.2,0.99)
\rput[bl](6.4,0.99){$S$}
\rput[bl](12.55,0.86){$p$}
\psbezier[linecolor=black, linewidth=0.04](7.6,0.19)(7.6,0.59)(7.6,0.79)(7.8,0.79)
\psbezier[linecolor=black, linewidth=0.04](9.0,0.19)(9.0,0.59)(9.0,0.59)(9.2,0.59)(9.4,0.59)(9.4,0.59)(9.4,0.19)
\psbezier[linecolor=black, linewidth=0.04](8.6,0.19)(8.6,0.59)(8.6,0.79)(8.4,0.79)
\psbezier[linecolor=black, linewidth=0.04](9.8,0.19)(9.8,0.59)(9.8,0.79)(10.0,0.79)
\psbezier[linecolor=black, linewidth=0.04, linestyle=dotted, dotsep=0.10583334cm](8.8,0.19)(8.8,-0.81)(9.6,-0.81)(9.6,0.19)
\psbezier[linecolor=black, linewidth=0.04](11.4,0.19)(11.4,0.59)(11.4,0.59)(11.6,0.59)(11.8,0.59)(11.8,0.59)(11.8,0.19)
\psbezier[linecolor=black, linewidth=0.04](11.0,0.19)(11.0,0.59)(11.0,0.79)(10.8,0.79)
\psbezier[linecolor=black, linewidth=0.04, linestyle=dotted, dotsep=0.10583334cm](11.2,0.19)(11.2,-0.81)(12.0,-0.81)(12.0,0.19)
\rput[bl](10.2,0.79){$\dots$}
\psline[linecolor=black, linewidth=0.04](7.8,0.79)(8.4,0.79)
\rput[bl](0.4,-0.91){$T_1$}
\rput[bl](3.0,-0.91){$T_2$}
\rput[bl](4.4,-0.91){$T_{2g-1}$}
\rput[bl](7.4,-0.91){$T_{2g}$}
\rput[bl](8.8,-1.11){$T_{2g+1}$}
\rput[bl](11.0,-1.11){$T_{2g+s-1}$}
\end{pspicture}
}
   \caption{A Seifert surface for a link $L$. All the components $T_1,\dots, T_{2g+s-1}$ might be linked, the dotted lines are drawn only to indicate the endpoints of each $T_i$.}
   \label{figure:Seifert}
\end{figure}

\def\FF{\overline{F}}

Let $X=V_{\la}\ot V_{\la}^*$. Then $X^*\cong V_{\la}^{**}\ot V_{\la}^*\cong X$ via the pivotal structure of $\CC$, denote by $j:X^*\to X$ such isomorphism. By (\ref{eq: pivotal on Verma}), this is given by 
\begin{align}
\label{eq: j on X}
    j((v_A\ot v_B')')=j(v''_B)\ot v'_A=q^{(r'-1)(\la-d(B),2\rho)}v_B\ot v'_A=q^{(r'-1)(\la,2\rho)}q^{n_B}v_B\ot v'_A
\end{align}
for some $n_B\in\Z$. 
\medskip

Suppose each component of $T$ is oriented to the right. Then $$F(T,X):\C\to  (X^*\ot X^*\ot X\ot X)^{\ot g}\ot (X^*\ot X)^{\ot s-1}.$$
Let $\FF(T,X)$ be the result of applying $j$ to each $X^*$ tensor factor:
$$\FF(T,X)=[(j\ot j\ot \id_X\ot \id_X)^{\ot g}\ot (j\ot\id_X)^{\ot s-1}](F(T,X)):\C\to X^{\ot 4g+2(s-1)}.$$
Then $F(L_o,V_{\la}):V_{\la}\to V_{\la}$ is given by
$$(\id_{V_{\la}}\ot \lev_{V_{\la}}^{\ot 4g+2(s-1)})(\FF(T,X)\ot\id_{V_{\la}}).$$

Now, $X\in \CC_{[0]}$ and by Corollary \ref{cor: unrolled at deg 0 is small quantum} we can think of $X$ as a $\uqg$-module and $\FC(T,X)$ as the Reshetikhin-Turaev invariant of $T$ colored by $X$ inside the category $\Rep(\uqg)$. As mentioned above, this is determined by the universal $\uqg$-invariant $$Z_T\in \uqg^{\ot 2g+s-1}$$ as follows: let $Z^{ev}_T\in \uqg^{\ot 4g+2(s-1)}$ be the element obtained from $Z_T$ by inserting the unit $1_{\uqg}$ in all odd tensor factors. Let $L_{Z^{ev}_T}$ be the endomorphism of $(X^*\ot X)^{\ot 2g+s-1}$ that is left multiplication by $Z^{ev}_T$, so the units $1_{\uqg}$ in $Z_T^{ev}$ act on the $X^*$ factors and each tensor factor of $Z_T$ acts on the corresponding $X$ factor. Then
\begin{align*}
    \FC(T,X)=[(\id_{X^*}\ot\tau_{X,X^*}\ot \id_{X})^{\ot g}\ot (\id_{X^*\ot X})^{\ot s-1}]\circ L_{Z^{ev}_T}\circ  \coev_X^{\ot 2g+s-1}
\end{align*}
where $\coev_X$ denotes the usual right coevaluation of vector spaces. We can then write 
\begin{align}
\label{eq: Z'TX}
   \FF(T,X)=[(\id_X\ot\tau_{X,X}\ot \id_X)^{\ot g}\ot \id_X^{\ot 2(s-1)}]\circ L_{Z^{ev}_T}\circ [(j\ot\id_X) \coev_X]^{\ot 2g+s-1}. 
\end{align}
To summarize this discussion:

\begin{proposition}
\label{prop: Seifert formula}
    If $S$ is a genus $g$ Seifert surface of a link $L$ of $s$ components built from a tangle $T$ as above, then the ADO invariant of $L$ is given by $\ADOgr'(L)=\lb L_o\rb$ where
    $$\lb L_o\rb\id_{V_{\la}}=(\id_{V_{\la}}\ot \lev_{V_{\la}}^{\ot 4g+2(s-1)})(\FF(T,X)\ot\id_{V_{\la}})$$
    and $\FF(T,X)$ is given by the above formula.
\end{proposition}


\subsection{Proof of Theorem \ref{theorem: ADO genus}}\label{subs: Proof of THM 1} Let $L$ be an unframed, oriented link. Let $S$ be a Seifert surface of $L$ and consider $L$ as a zero-writhe framed link. Consider the formula for $\FF(T,X)$ given in (\ref{eq: Z'TX}). By (\ref{eq: j on X}), it is clear that $[(j\ot\id_X)\coev_X(1)]^{\ot 2g+s-1}$ is $q^{(r'-1)(\la,2\rho)(2g+s-1)}$ times a $\Q(q)$-linear combination of standard basis vectors of $(X\ot X)^{2g+s-1}$. By Proposition \ref{lemma: ESSENTIAL LEMMA}, applied to each of the $2g+s-1$ non-trivial tensor slots of $Z_T^{ev}$, $L_{Z_T^{ev}}$ acts on the standard basis of $(X\ot X)^{2g+s-1}$ with coefficients polynomials in $\Q(q)[q^{-2\la_1},\dots, q^{-2\la_n}]$ of bounded degree $$\deg_{q^{-2\la_i}}\leq (2g+s-1)(r'-1)k_i$$ for every $i=1,\dots,n$. Therefore, in the standard basis, $\FF(T,X)(1) \in (X\ot X)^{2g+s-1}$ has coefficients of the form $$q^{(r'-1)(\la,2\rho)(2g+s-1)} p(q^{-2\la_1},\dots, q^{-2\la_n})$$
where $p(q^{-2\la_1},\dots, q^{-2\la_n})$ is a polynomial satisfying the above degree bound. Since $\ev_{V_{\la}}$ only takes values $0$ or 1 in the standard basis, it follows from Proposition \ref{prop: Seifert formula} that $$\ADOgr'(L)=\ADOgr(L)=\lb L_o\rb=q^{(r'-1)(\la,2\rho)(2g+s-1)} P(q^{-2\la_1},\dots, q^{-2\la_n})$$
where $P$ is also a polynomial in $\Q(q)[q^{-2\la_1},\dots, q^{-2\la_n}]$ still satisfying the same degree bound on each $q^{-2\la_i}$.\\

Now setting $t_i=q^{2\la(H_i)}$, the above is written as 
$$\ADOgr'(L)= t_1^{\frac{1}{2}(r'-1)k_1(2g+s-1)}\cdots t_n^{\frac{1}{2}(r'-1)k_n(2g+s-1)}P(t_1^{-1},\dots, t_n^{-1}).$$
In other words, up to the overall fractional power stated in the theorem, $\ADOgr'(L)$ becomes a polynomial in $\Q(q)[t_1^{\pm 1},\dots, t_n^{\pm 1}]$ satisfying $\deg_{t_i} \ \ADOgr'(L)\leq (2g+s-1)(r'-1)k_i$ for every $i$. Since the Seifert surface $S$ was arbitrary, this proves the theorem.

\subsection{Additional remarks} There are two situations in which $\ADOgr'(L)$ is a polynomial in $\Q(q)[t_1^{\pm 1},\dots, t_n^{\pm 1}]$ for all links $L$:
\begin{enumerate}
    \item When $r'$ is odd,
    \item When all $k_i$'s above are even. This happens for instance for $\sl_{n+1}$ with even $n$ since $2\rho=\sum_{i=1}^n i(n+1-i)\a_i$. For $\sl_3$ at $r=4$ this was noted by Harper in \cite[Corollary 6.8]{Harper:sl3-invariant} (though his $t_i^2$ is our $t_i$).
\end{enumerate}
Otherwise, $\ADOgr'(L)\in \Q(q)[t_1^{\pm 1},\dots, t_n^{\pm 1}]$ for links with odd number of components.

\subsection{Computations} 
\label{subs: computations KT and C}

\def\Harp{P}
We give some computations here for $ADO'_{\sl_3,4}$ illustrating our theorem. These computations were also carried out by Harper \cite{Harper:sl3-invariant}. Denote $x=t_1$ and $y=t_2$. Note that for $\sl_3$, $\sum_{\a\in\Dep}\a=2(\a_1+\a_2)$. Since $r'=2$, our theorem states that, if $K$ is a knot, then $P(K,x,y)=\ADO_{\sl_3,4}(K)$ has degree in $x$ less than $4g(K)$ and similarly for $y$. It turns out that $\Harp(K,x,y)=\Harp(K,x^{-1},y^{-1})$ so the maximal positive power of $x$ or $y$ in $\Harp(K,x,y)$ is bounded above by $2g(K)$.\\

\begin{center}
\begin{tabular}{ c|c }
 \hline
 $K$ & $ADO'_{\sl_3,4}(K,x,y)$  \\ 
 \hline
 $3_1$ & $1 + x^{-2} - 2x^{-1} - 2 x + x^2 + y^{-2} + x^{-2} y^{-2} - 
 x^{-1} y^{-2} - 2y^{-1} - x^{-2} y^{-1} $\\
 & $+ 2x^{-1} y^{-1} + xy^{-1} - 2 y + yx^{-1} + 2 x y - 
 x^2 y + y^2 - x y^2 + x^2 y^2 $  \\
 \hline
 $4_1$ & $25 + x^{-2} - 12x^{-1} - 12 x + x^2 + y^{-2} + x^{-2} y^{-2} - 3
 x^{-1} y^{-2} - 12y^{-1} - 3x^{-2} y^{-1} $\\
& $ + 12x^{-1} y^{-1} + 3 xy^{-1} - 12 y + 3 yx^{-1} + 
 12 x y - 3 x^2 y + y^2 - 3 x y^2 + x^2 y^2$ \\
 \hline
$5_1$ & $1 + 1/x^4 + 2/x^3 + 1/x^2 + x^2 + 2 x^3 + x^4 + 1/y^4 + 1/(
 x^4 y^4) + 1/(x^3 y^4) + 1/(x^2 y^4)  $ \\
 & $+ 1/(x y^4) + 2/y^3+ 1/(x^4 y^3) + 2/(x^3 y^3) + 2/(x^2 y^3)+ 2/(x y^3) + x/y^3 + 1/y^2 $\\
  & $ + 1/(x^4 y^2) + 2/(x^3 y^2)+ 1/(x^2 y^2) + 1/(x y^2) + (2 x)/y^2 + x^2/y^2 + 1/(x^4 y) $\\
 & $ + 2/(
 x^3 y) + 1/(x^2 y) + x/y + (2 x^2)/y + x^3/y + y/x^3+ (2 y)/x^2 + y/x + x^2 y    $\\
 & $+ 2 x^3 y + x^4 y+ y^2 + y^2/x^2+ (
 2 y^2)/x + x y^2 + x^2 y^2 + 2 x^3 y^2 + x^4 y^2+ 2 y^3   $\\
 & $+ y^3/x + 2 x y^3+ 2 x^2 y^3+ 2 x^3 y^3 + x^4 y^3 + y^4 + x y^4 + x^2 y^4 + x^3 y^4 + x^4 y^4$\\
 \hline
 $5_2$ & $37 + 6x^{-2} - 26x^{-1} - 26 x + 6 x^2 + 6y^{-2} + 6x^{-2} y^{-2} - 10
 x^{-1} y^{-2} - 26y^{-1}  - 10x^{-2} y^{-1}$\\
& $ + 26x^{-1} y^{-1} + 10 xy^{-1} - 26 y + 10 yx^{-1} + 
 26 x y - 10 x^2 y + 6 y^2 - 10 x y^2 + 6 x^2 y^2  $ \\
 \hline
$6_1$& $  85 + 6x^{-2} - 46x^{-1} - 46 x + 6 x^2 + 6y^{-2} + 6x^{-2} y^{-2} - 14
 x^{-1} y^{-2} - 46y^{-1} - 14x^{-2} y^{-1}$\\
 & $ + 46x^{-1} y^{-1} + 14 xy^{-1} - 46 y + 14 yx^{-1} + 
 46 x y - 14 x^2 y + 6 y^2 - 14 x y^2 + 6 x^2 y^2 $ \\
 \hline
$6_2$ & $25 + 1/x^4 - 12/x^3 + 35/x^2 - 38/x - 38 x + 35 x^2 - 
 12 x^3 + x^4 + 1/y^4 + 1/(x^4 y^4)  $\\
& $- 3/(x^3 y^4) + 3/(x^2 y^4) - 3/(x y^4) - 12/y^3 - 3/(x^4 y^3) + 12/(x^3 y^3) - 18/(x^2 y^3)   $\\
 & $+ 18/(
 x y^3)+ (3 x)/y^3 + 35/y^2 + 3/(x^4 y^2) - 18/(x^3 y^2)+ 35/(
 x^2 y^2) - 37/(x y^2) $\\
 & $- (18 x)/y^2 + (3 x^2)/y^2 - 38/y - 3/(
 x^4 y) + 18/(x^3 y) - 37/(x^2 y) + 38/(x y)  $\\
 & $+ (37 x)/y - (18 x^2)/y + (3 x^3)/y- 38 y + (3 y)/x^3 - (18 y)/x^2 + (37 y)/x + 
 38 x y   $\\
& $+ (3 y^2)/x^2 - (18 y^2)/x - 37 x y^2+ 35 x^2 y^2 - 18 x^3 y^2 + 3 x^4 y^2 - 
 12 y^3 + (3 y^3)/x $\\
& $+ 18 x y^3 - 18 x^2 y^3 + 12 x^3 y^3 - 
 3 x^4 y^3 + y^4   - 3 x y^4 + 3 x^2 y^4 - 3 x^3 y^4 + x^4 y^4$\\
 \hline
\end{tabular}
\end{center}

The knots $3_1,4_1,5_2,6_1$ have genus 1, which explains why the above polynomials are relatively simple for such knots. All positive powers of $x$ and $y$ are $\leq 2g=2$ as expected by our theorem. The knots $5_1, 6_2$ have genus 2, this is why they are a bit longer. Again, one sees that every positive power of $x$ or $y$ is $\leq 4$. \\

 We also carried the computation of $\Harp(K)$ for the Kinoshita-Terasaka knot KT and the Conway knot C:

\begin{align*}
   \Harp(KT)&=721 - 12/x^3- 12/y^3-  12 x^3  + 12/(x^3 y^3)+  12 x^3 y^3-12 y^3\\
  &-34/(x^3 y^2)- 34 x^3 y^2- 34/(x^2 y^3)- 34 x^2 y^3\\
  & + 34/( x y^3)+ 34 x y^3+ 34/(x^3 y)+ 34 x^3 y  \\
 &- (34 x)/y^2- (34 y^2)/x- (34 x^2)/y- (34 y)/x^2 \\
 &+ 148/x^2 + 148 x^2+ 148/y^2+ 148/(x^2 y^2)+ 148 y^2+ 148 x^2 y^2 \\
 &  - 228/(x y^2)    - 228/(x^2 y) - 228 x^2 y  - 228 x y^2+ (228y)/x+ (228 x)/y   \\
 &- 496/x - 496 x+ 496/(x y)- 496/y - 496 y   + 496 x y\\
 \end{align*}
As observed, the maximal positive degree in $x$ is 3, hence $3\leq 2g(K)$ so that $g(K)>1$. A genus 2 Seifert surface is shown in \cite{Gabai:genera}, hence it has minimal genus.\\

For the Conway knot: 
\begin{align*}
 &\Harp(C)=649 + 2/(x^4 y^6)+ 2/(x^2 y^6) + 2/(x^6 y^4)+ (2 x^2)/y^4 + 2/(x^6 y^2)\\
 &+ (2 x^4)/y^2 + (2 y^2)/x^4+ 2 x^6 y^2  + (2 y^4)/x^2 + 2 x^6 y^4+ 2 x^2 y^6 + 2 x^4 y^6\\
 & + 4/x^5  + 4/x^3 +  4 x^3 + 4 x^5 - 4/(x^3 y^6) + 4/y^5 - 4/(x^5 y^5) + 4/(x^4 y^5) - 4/(x y^5)+ 4/(x^5 y^4)  \\
&  + (4 x)/y^4+ 4/y^3- 4/(x^6 y^3) - 4/(x^3 y^3) + (4 x^3)/y^3- 4/(x^5 y)  + (4 x^4)/y + (4 y)/x^4 - 4 x^5 y\\
& + 4 y^3 + (4 y^3)/x^3 - 4 x^3 y^3- 4 x^6 y^3 + ( 4 y^4)/x+ 4 x^5 y^4 + 4 y^5- 4 x y^5+ 4 x^4 y^5- 4 x^5 y^5 - 4 x^3 y^6 \\
 & + 8/(x^3 y^5) - 8/(x^2 y^5)+ 8/(x^3 y^4) + 8/(x y^4)  + 8/(x^5 y^3) + 8/(x^4 y^3)- (8 x)/y^3- (8 x^2)/y^3 \\
&  - 8/(x^5 y^2)  - (8 x^3)/y^2+ 8/(x^4 y)- (8 x^3)/y - (8 y)/x^3+ 8 x^4 y- (8 y^2)/x^3- 8 x^5 y^2 \\
 & - ( 8 y^3)/x^2- (8 y^3)/x+ 8 x^4 y^3 + 8 x^5 y^3+ 8 x y^4+ 8 x^3 y^4 - 8 x^2 y^5 + 8 x^3 y^5 \\
& + 12/(x^2 y^4) + 12/( x^4 y^2) + (12 x^2)/y^2 + (12 y^2)/x^2 + 12 x^4 y^2+ 12 x^2 y^4\\
&  - 20/x^4 - 20 x^4 - 20/y^4- 20/(x^4 y^4) - 20 y^4 - 20 x^4 y^4\\
& - 46/(x^2 y^3) + 46/(x y^3)- 46/(x^3 y^2) - (46 x)/y^2+ 46/(x^3 y)- (46 x^2)/y- (46 y)/x^2 \\
&+ 46 x^3 y- (46 y^2)/x- 46 x^3 y^2+ 46 x y^3 - 46 x^2 y^3\\
&+ 164/x^2 + 164 x^2  + 164/y^2+ 164/(x^2 y^2) + 164 y^2  + 164 x^2 y^2 \\
&- 248/(x y^2) - 248/(x^2 y)  + ( 248 x)/y + (248 y)/x- 248 x^2 y- 248 x y^2\\
   &- 476/x - 476 x  - 476/y + 476/(x y) - 476 y + 476 x y\\
\end{align*}

The maximal positive degree in $x$ is $6$, hence $6\leq 2g(K)$ or $g(K)\geq 3$. A genus 3 Seifert surface is given in \cite{Gabai:genera}, hence it has minimal genus.
\medskip

It is interesting to note that the above two polynomials become equal at $x=y$:
\begin{align*}
   \Harp(KT)=\Harp(C)&= 1177 + 12/y^6 - 68/y^5 + 216/y^4 - 480/y^3 + 792/y^2 - 1060/y \\
   &-  1060 y + 792 y^2 - 480 y^3+ 216 y^4 - 68 y^5 + 12 y^6.
\end{align*}

This shows the power of multivariable knot invariants as opposed to single-variable.

\subsection{The non-simply-laced case}\label{subs: non-simply-laced} In the non-simply-laced case, there are $d_1,\dots,d_n$ making $(d_ia_{ij})$ symmetric and the form $(-,-)$ on $\hh^*$ is defined by $(\a_i,\a_j)=d_ia_{ij}$. Setting $q_i=q^{d_i}$, the definition of the unrolled quantum group is adapted by setting $$K_iE_j=q_i^{a_{ij}}E_jK_i, \ K_iF_j=q_i^{-a_{ij}}F_jK_i, \ [E_i,F_j]=\d_{ij}(K_i-K_i^{-1})/(q_i-q_i^{-1})$$ and using the $q$-integers $[k]_{q_i}=(q_i^k-q_i^{-k})/(q_i-q_i^{-1})$ in the Serre relations, see \cite{CP:BOOK}. If $r_{\b}$ is the order of $q^{(\b,\b)}$ for every $\b\in\Dep$, then $\UHr$ and $\uqg$ are defined in a similar way as before but setting $E_{\b}^{r_{\b}}=0$ for all $\b\in\Dep$. Weight modules must satisfy $K_i=q_i^{H_i}$ for every $i$. Then, Proposition \ref{lemma: ESSENTIAL LEMMA} states that for every $z\in\uqg$ the coefficients of $z:X\to X$ in the standard basis are polynomials in $\Q(q)[q_1^{-2\la_1},\dots,q_n^{-2\la_n}]$ where the power of $q_i^{-2\la_i}$ is bounded above by $R_i$, where
\begin{align}
    \label{eq: non simply-laced ESSENTIAL LEMMA}
    \sum_{\b\in\Dep}(r_{\b}-1)\b=R_1\a_1+\dots+R_n\a_n.
\end{align}

\begin{theorem}
\label{theorem: non-simply-laced Thm 1}
Let $L$ be a link of $s$ components. There is an isotopy invariant $$\ADOgr'(L)\in (t_1^{R_1}\cdots t_n^{R_n})^{\frac{s-1}{2}}\cdot \Q(q)[t_1^{\pm 1},\dots, t_n^{\pm 1}]$$
satisfying
$$\deg_{t_i}\ADOgr'(L)\leq (2g(L)+s-1)R_i$$
for every $i=1,\dots,n$, where $R_1,\dots, R_n$ are defined by (\ref{eq: non simply-laced ESSENTIAL LEMMA}).
\end{theorem}

The invariant and the theorem is obtained the same way as in the simply-laced case, but setting $t_i=q_i^{2\la_i}$ for every $i$ instead and applying the corresponding generalization of Proposition \ref{lemma: ESSENTIAL LEMMA} with (\ref{eq: non simply-laced ESSENTIAL LEMMA}) above.


\section{Unrolled restricted quantum groups as doubles of crossed products}
\label{section: equivariantization thm}

In this section we define all the elements appearing in Theorem \ref{theorem: equivariantization}: crossed products, relative Drinfeld centers, equivariantization. Our proof relies on Hopf algebra descriptions of these objects, for this we introduce twisted Drinfeld doubles and semi-direct products. The proof of Theorem \ref{theorem: equivariantization} comprises Subsections \ref{subs: TDD of small Borel} to \ref{subs: ribbon for TDDs} and is summarized in Subsection \ref{subs: proof of Theorem 2}. Finally, we prove Theorem \ref{corollary: P coincides with ADO} in Subsection \ref{subs: recovering P}.\\

During the whole section we let $G$ be an abelian group.

\def\CC{\mathcal{C}}

\def\wDe{\widetilde{\De}}\def\wot{\widetilde{\ot}}\def\DD{\mathcal{D}}
\def\v{\phi}

\def\uD{\underline{D}}

\subsection{$G$-graded categories}  A monoidal category $\CC$ is {\em $G$-graded } if it splits as a disjoint union of full subcategories $\CC=\coprod_{\la\in G}\CC_{\la}$ such that $\CC_{\la}\ot \CC_{\mu}\sb \CC_{\la+\mu}$. We say that $\CC$ is {\em $G$-crossed} if $G$ acts on $\CC$ by monoidal autoequivalences $\v_{\la}:\CC\to \CC$ such that $\v_{\la}(\CC_{\mu})\sb \CC_{\mu}$ for each $\la,\mu\in G$. We will only consider {\em strict} monoidal actions, that is, the isomorphisms $\v_{\la+\mu}(X)\cong \v_{\la}(\v_{\mu}(X))$  and $\v_{\la}(X\ot Y)\cong \v_{\la}(X)\ot \v_{\la}(Y)$ are all identity maps of vector spaces (our categories will be representation categories of Hopf group-coalgebras).
A {\em braided $G$-crossed category} is a $G$-crossed category with natural isomorphisms $c_{X,Y}:X\ot Y\to \v_{\la}(Y)\ot X$ for all $X\in \CC_{\la}, Y\in\CC, \la\in G$ satisfying certain axioms. For more details, see \cite{Turaev:homotopy}. 
\medskip

\begin{remark}
    If $G$ is abelian, a $G$-graded category $\CC$ might be braided in two senses. On the one hand, $\CC$ might carry a $G$-action and be $G$-braided in the above sense. On the other hand, it might be braided in the usual sense (no $G$-action involved). 
\end{remark}

\begin{remark}
\label{remark: bicharacter}
    Suppose $\CC$ is $G$-graded and braided in the usual sense. If $\chi:G\t G\to \C^*$ is a bicharacter, then it is easy to see that $c^{\chi}_{X,Y}=\chi(\la,\mu)c_{X,Y}$ is a braiding, where $X\in \CC_{\la}, Y\in \CC_{\mu}$. If $(\CC,c,\theta)$ is ribbon, then $(\CC,c^{\chi},\theta^{\chi})$ is ribbon where $\theta^{\chi}_X=\chi(\la,\la)\theta_X$ if $X\in\CC_{\la}$. It is easy to see that the unframed link invariants obtained from $(\CC,c^{\chi})$ are the same as those of $(\CC,c)$: each crossing multiplies the invariant by a $\chi(\la,\la)^{\pm 1}$, but these cancel out when multiplying by the framing factor.
\end{remark}

Suppose $G$ acts on monoidal category $\CC$ by monoidal autoequivalences $\phi_{\la}:\CC\to\CC$. Then we define the {\em crossed product} $\CC\rtimes G$ as the category whose objects are pairs $(X,\la)$ with $X\in \CC$ and $\la\in G$ and $\Hom_{\CC\rtimes G}((X,\la),(Y,\mu))=\Hom_{\CC}(X,Y)\d_{\la,\mu}$. This is a $G$-graded category in a natural way. It is a $G$-graded monoidal category if we define $$(X,\la)\ot (Y,\mu)=(X\ot\phi_{\la}(Y),\la+\mu).$$
\medskip

The {\em relative Drinfeld center} of $\CC\rtimes G$ is the category consisting of pairs $(X,\la,\s)$ where $(X,\la)$ is an object of $\CC\rtimes G$ and $\s=\{\s_Y\}_{Y\in\CC}$ is a {\em half-braiding} of $(X,\la)$, that is, a collection of isomorphisms $$\s_Y:(Y,0)\ot (X,\la)\to (X,\la)\ot (Y,0)$$
satisfying certain axioms. We denote it by $\ZZ_{\CC}(\CC\rtimes G)$, this is a braided $G$-crossed category in a natural way, see \cite{GNN:centers} or \cite{LN:TDD} for more details.

\medskip

\subsection{Equivariantization} The {\em equivariantization} of a $G$-crossed category $\CC$ is the category $\CC^G$ whose objects are pairs $(V,f)$ where $V\in \CC$ and $f=\{f_{\la,V}\}_{\la\in G}$ is a family of isomorphisms $f_{\la,V}:\v_{\la}(V)\to V$ satisfying $f_{\la+\mu,V}=f_{\la,V}\circ \v_{\la}(f_{\mu,V})$. It is easy to see that if $\CC$ is a braided $G$-crossed category, then $\CC^G$ is braided in the usual sense with braiding $c_{X,Y}^G$ defined as the composition

\begin{align}
\label{eq: braiding on equivariantization}
\xymatrix{c^G_{X,Y}:X\ot Y\ar[r]^{c_{X,Y}} & \phi_{\la}(Y)\ot X\ar[r]<2pt>^{f_{\la, Y}\ot\id} & Y\ot X}
\end{align}

provided $X\in \CC_{\la}$. See \cite{EGNO:BOOK} for more details.

\subsection{Hopf group-coalgebras} A {\em Hopf $G$-coalgebra} is a family $\uD=\{D_{\la}\}_{\la\in G}$ of algebras with unit $D_{\la}$ together with coproducts $\De_{\la,\mu}:D_{\la+\mu}\to D_{\la}\ot D_{\mu}$, a counit $\e:D_0\to\C$ and antipodes $ S_{\la}:D_{\la}\to D_{-\la}$ for each $\la,\mu\in G$ satisfying a generalization of the usual Hopf algebra axioms. From the axioms it follows that $\Rep(\uD)=\coprod_{\la\in G}\Rep(D_{\la})$ is a $G$-graded monoidal category with tensor product induced by $\De_{\la,\mu}$. The antipode induces a duality on this category. See \cite{Turaev:homotopy, Virelizier:Hopfgroup} for more details.
\medskip

We say that $\uD$ is {\em $G$-crossed} if $G$ acts on each $D_{\la}$ by algebra automorphisms $\phi_{\nu}$ and $(\phi_{\nu}\ot\phi_{\nu})\De_{\la,\mu}=\De_{\la,\mu}\phi_{\nu}$ for each $\la,\mu,\nu\in G$. It follows that $\Rep(\uD)$ is $G$-crossed if for each $X\in\Rep(D_{\la})$ we define $\phi_{\nu}(X)$ to be $X$ as a vector space but with $D_{\la}$-module structure $h\cdot_{\nu}x:=\phi_{\nu}(h)x$ ($h\in D_{\la}, x\in X$).\\

We say $\uD$ is {\em $G$-braided} if it is endowed with a family $R=\{R_{\la,\mu}\}_{\la,\mu\in G}$ of invertible elements of $D_{\la}\ot D_{\mu}$ with the following property: if $X\in\Rep(D_{\la}), Y\in\Rep(D_{\mu})$ then $$c_{X,Y}: X\ot Y\to \v_{\la}(Y)\ot X, \ x\ot y\mapsto \tau(R_{\la,\mu}\cdot x\ot y)$$
defines a $G$-braiding on $\Rep(\uD).$ Here $\tau$ denotes the usual transposition of tensor factors.\\

\subsection{Drinfeld twists}
\label{subs: Drinfeld twist}
A {\em Drinfeld twist} on a Hopf $G$-coalgebra $\uD$ consists of a family $J=\{J_{\la,\mu}\}$ of invertible elements of $D_{\la}\ot D_{\mu}$ satisfying $$(1_{\la}\ot J_{\mu,\nu})\cdot[\id\ot\De_{\mu,\nu}(J_{\la,\mu+\nu})]=(J_{\la,\mu}\ot 1_{\nu})\cdot[\De_{\la,\mu}\ot\id(J_{\la+\mu,\nu})].$$
 This guarantees that $\De^J_{\la,\mu}:=J_{\la,\mu}\De_{\la,\mu}J^{-1}_{\la,\mu}$ is a new associative coproduct on $\uD$. If $\uD$ is $G$-braided and $J$ is a Drinfeld twist, then $(\uD,\De^J)$ is $G$-braided with
 $$ R^J_{\la,\mu}=\tau(J_{\mu,\la})\cdot R_{\la,\mu}\cdot J_{\la,\mu}^{-1}.$$

Applying a Drinfeld twist does not change the associated monoidal categories up to equivalence. This requires the more general notion of monoidal equivalence: a {\em monoidal functor} $F:(\CC,\ot)\to(\DD,\ot')$ between monoidal categories $(\CC,\ot), (\DD,\ot')$ is a functor $F:\CC\to\DD$ together with a family of isomorphisms $$J_{X,Y}:F(X\ot Y)\to F(X)\ot' F(Y)$$ satisfying hexagon axioms (see \cite{EGNO:BOOK}), and two monoidal categories are {\em monoidally equivalent} if there are monoidal functors $F:\CC\to\DD$ and $G:\DD\to\CC$ which are inverses to each other. If $(\CC,c),(\DD,c')$ are braided, a monoidal equivalence is braided if $c'_{F(X),F(Y)}J_{X,Y}=J_{Y,X}F(c_{X,Y})$.\\

In the case $\CC=\Rep(\uD,\De)$ and $\DD=\Rep(\uD,\De^J)$, the identity functor $F:\Rep(\uD,\De)\to \Rep(\uD,\De^J)$ is a monoidal equivalence with $J_{X,Y}(x\ot y)=J_{\la,\mu}\cdot x\ot y$ if $X\in \Rep(D_{\la}), Y\in\Rep(D_{\mu})$. If $\uD$ is braided and $\uD^J$ has the above $R$-matrix, then the equivalence is braided. Thus:

\begin{lemma}
    \label{lemma: Drinfeld twist gives equivalence}
    If $J=\{J_{\la,\mu}\}$ is a Drinfeld twist on a $G$-braided Hopf $G$-coalgebra $\uD$, then the categories $\Rep(\uD)$ and $\Rep(\uD^J)$ are equivalent as $G$-braided categories.
\end{lemma}

\subsection{Semidirect product and equivariantization} Let $\{D_{\la}\}_{\la\in G}$ be a $G$-braided Hopf $G$-coalgebra. Since $G$ acts on each $D_{\la}$ we can define the semidirect product $D^G_{\la}:=\C[G]\ltimes D_{\la}$, thus, in $D^G_{\la}$ one has $\mu\cdot x=\phi_{\mu}(x)\cdot\mu$ for each $\mu\in G, x\in D_{\la}$. Then $\uD^G=\{D^G_{\la}\}_{\la\in G}$ is a Hopf $G$-coalgebra with the structure tensors of $\uD$ and the usual Hopf structure on $\C[G]$. The category $\Rep(\uD^G)$ is braided (in the usual sense) with 
\begin{align}
    c_{X,Y}:X\ot y\to Y\ot X, \ x\ot y\mapsto \tau(R^G_{\la,\mu}\cdot x\ot y)
\end{align}
if $X\in \Rep(D_{\la})^G$, $Y\in\Rep(D_{\mu}^G)$ and $R^G_{\la,\mu}=(1\ot \phi_{\la})\cdot R_{\la,\mu}$.

\begin{proposition}
\label{prop: equivariantization is SS}
The equivariantization $\Rep(\uD)^G$ is braided equivalent to $\Rep(\uD^G)$.
\end{proposition}
\begin{proof}
Let $(V,f)\in \CC^G$, say $V\in \CC_{\la}$. Since $\v_{\mu}(V)=V$ as vector spaces, we can define a $G$-action on $V$ by $\mu\cdot v=f_{\mu}(v)$. To say that $f_{\mu}:\v_{\mu}(V)\to V$ is a morphism is equivalent to $f_{\mu}(\v_{\mu}^{-1}(h)v)=hf_{\mu}(v)$ for each $h\in D_{\la}$ or, in terms of the $G$-action, $\mu\cdot (\v_{\mu}^{-1}(h)v)=h (\mu\cdot v)$. This is equivalent to $V$ being a module over $\C[G]\ltimes D_{\la}$. It is easy to see that the braiding $c^G_{X,Y}$ defined in (\ref{eq: braiding on equivariantization}) is left multiplication by $R^G_{\la,\mu}$, proving that the equivalence is braided.
\end{proof}

\subsection{Twisted Drinfeld doubles}
\label{subs: TDDs}
Suppose $G=(\C^n,+)$ acts over a finite dimensional Hopf algebra $H$ by Hopf automorphisms $\phi_{\la}$. We define a $G$-braided Hopf $G$-coalgebra as follows: as a vector space, let $D(H)_{\la}=H^*\ot H$ for each $\la\in G$. The multiplication is given by $$(p\ot a)(q\ot b)=\lb q_{(1)},S^{-1}(\phi_{-\la}(a_{(3)})) \rb\lb q_{(3)},a_{(1)}\rb p\cdot q_{(2)}\ot a_{(2)}\cdot b$$
for $a,b\in H, p,q\in H^*$, where $\lb-, -\rb$ denotes the usual vector space pairing. The coproduct is given by $$\De_{\la,\mu}=(\id_{H^*}\ot\id_H\ot\id_{H^*}\ot \phi_{\la}^{-1})\circ \De_{D(H)}.$$
The antipode is $$S_{\la}(p\ot a)=(\e\ot S(\phi_{-\la}(a)))\cdot (p\circ S^{-1}\ot 1).$$
The group $G$ acts on each $D(H)_{\la}$ by $(\phi_{\la}^{-1})^*\ot \phi_{\la}$. This is $G$-braided with $R$-matrix $$R_{\la,\mu}=\sum \phi_{\la}(h_i)\ot h^i$$
where $(h_i)$ is a basis of $H$ and $(h^i)\sb H^*$ is the dual basis. With this structure, we call $\{D(H)_{\la}\}_{\la\in G}$ the {\em twisted Drinfeld double} of $H$ and denote it by $\uDH$. The following is shown in \cite{LN:TDD}.

\begin{proposition}
\label{prop: rel Drinfeld and TDD}
The category $\coprod_{\la\in G} \Rep(D(H)_{\la})$ is equivalent to the relative Drinfeld center $\ZZ_{\Rep(H)}(\Rep(H)\rtimes G)$ as a braided $G$-crossed category.
\end{proposition}

\def\uA{\underline{A}}\def\uwtA{\underline{{A^J}}}

Now let $A_{\la}=D(H)_{\la}^G=\C[G]\ltimes D(H)_{\la}$, and $\uA=\uD^G$. It is easy to see that, in the semidirect product $$J_{\la,\mu}=(1_{A_{\la}}\ot\phi_{\la/2})\in A_{\la}\ot A_{\mu}$$
defines a Drinfeld twist on $\uA$. Thus, we get a Drinfeld twisted semi-direct product $\uA^J$ which is braided with $R$-matrix becomes
\begin{align}
    \label{eq: R-matrix in AJ}
    \begin{split}
R^J_{\la,\mu}&=\tau(J_{\mu,\la})R^G_{\la,\mu}J_{\la,\mu}^{-1}\\
    &=\phi_{\mu/2}\ot 1\cdot 1\ot \phi_{\la}\cdot R_{\la,\mu}\cdot 1\ot\phi_{-\la/2}\\
    &=\phi_{\mu/2}\ot\phi_{\la/2}\cdot(\id\ot\phi_{\la/2})(R_{\la,\mu}).
    \end{split}
\end{align}
\def\AJ{\wt{A}}

\begin{proposition}
    The category $\Rep(\uwtA)$ is braided equivalent to the equivariantization $\ZZ_{\Rep(H)}(\Rep(H)\rtimes G)^G.$
\end{proposition}
\begin{proof}
    Indeed, $\Rep(\uwtA)\cong\Rep(\uA)$ as braided categories as explained in Subsection \ref{subs: Drinfeld twist}, so the result follows from Propositions \ref{prop: equivariantization is SS} and \ref{prop: rel Drinfeld and TDD}.
\end{proof}

\subsection{TDD of the small quantum Borel}
\label{subs: TDD of small Borel}

\def\AJ{\wt{A}}
\def\Ea{E_{\a}}\def\Fa{F_{\a}}

\def\EE{\boldsymbol{E}}\def\KK{\boldsymbol{K}}\def\FF{\boldsymbol{F}}\def\KKK{\ov{\mathcal{K}}}
\def\wtKKK{\wt{\mathcal{K}}}

\def\Ula{\mathfrak{u}_{\la}}
\def\Umu{\mathfrak{u}_{\mu}}

\def\e{e}
\def\E{e}\def\F{f} \def\K{k}
\def\f{f}
Let $\gg,\hh,q,r$ be as in the beginning of Section \ref{section: unrolled}, in particular $\gg$ is simply-laced and $r$ coprime to the determinant of the Cartan matrix. Let $H=\uqb$ be the Borel part of the small quantum group $\uqg$. We will denote its usual generators by $\K_i,\E_i, i=1,\dots,n$. Then, $\hh^*\cong \C^n$ acts on $\uqb$ by Hopf automorphisms given by $\phi_{\la}(\E_i)=q^{2\la_i}\E_i$ and $\phi_{\la}(\K_i)=\K_i$ for each $\la\in\hh^*$. Here we denote $\la_i=\la(H_i)$ as usual. In \cite{Virelizier:Graded-QG}, it is shown that $D(\uqb)_{\la}$ is isomorphic to the algebra generated by $\K_i,\K'_i,\E_i,\F_i$, $i=1,\dots,n$ and with relations
\begin{align*}
\K_i\e_j&=q^{a_{ij}}\e_j\K_i, & \K_i\f_j&=q^{-a_{ij}}\f_j\K_i & \K_i\K_j&=\K_j\K_i, \\
\K'_i\e_j&=q^{-a_{ij}}\e_j\K'_i, & \K'_i\f_j&=q^{a_{ij}}\f_j\K'_i & \K'_i\K'_j&=\K'_j\K'_i, \\
[\e_i,\f_j]&=\d_{ij}\frac{\K_i-q^{-2\la_i}\K'_i}{q-q^{-1}}, & \K_i\K'_j&=\K'_j\K_i, & \K_i^{r}&=1, \\
\e_{\a}^{r'}&=0, & \f_{\a}^{r'}&=0, & (\K'_i)^r&=1 \\
\end{align*}
for $i,j=1,\dots,n$ and all $\a\in\Dep$. We will denote by $\Ula$
the quotient of $D(\uqb)_{\la}$ by the ideal generated by $\K_i\K'_i-1$ for $i=1,\dots,n$. The coproduct $\De_{\la,\mu}:\uu_{\la+\mu}\to \Ula\ot\Umu$ is given by
\begin{align*}
\De(\e_i)&=\e_i\ot\ \K_i+q^{-2\la_i}1\ot \e_i, & \De(\F_i)&=\K_i^{-1}\ot \f_i+\f_i\ot 1, & \De(\K_i)&=\K_i\ot \K_i.
\end{align*}

 Then, $\{\Ula\}_{\la\in\hh^*}$ is $\hh^*$-braided with $R_{\la,\mu}=\phi_{\la}\ot\id(\pi(h_i\ot h^i))$ where $\pi:D(\uqb)_{\la}\to\Ula$ is the projection. It is well-known that 
 $\pi(h_i\ot h^i)$ is the usual $R$-matrix of $\uqg$, which is $R_0=\KKK_0\cdot \TT_0$ where $\KKK_0$ and $\TT_0$ are as given in Subsection \ref{subs: braiding on weight}: $$\KKK_0=\frac{1}{r^n}\sum_{\a,\b\in Q_r}q^{-(\a,\b)}\K_{\a}\ot \K_{\b}, \ \ \TT_0=\sum_I c_I\cdot \E^I\ot \F^I.$$

\def\Urqg{\ov{U}_q(\gg)}

\subsection{New generators}
In $\Ula$, let $\KK_i=q^{\la_i}\K_i, \EE_i=q^{\la_i}\E_i, \FF_i=\F_i$. Then the only of the above relations that change are $$[\EE_i,\FF_i]=\frac{\KK_i-\KK_i^{-1}}{q-q^{-1}}$$
and $\KK_i^r=q^{r\la_i}$. Thus, we see that $\Ula$ is simply the quotient of $\Urqg$ by the relation $\KK_i^r=q^{r\la_i}$ for each $i$. Thus, there are root vectors $\EE_{\a}, \FF_{\a}$ in $\Ula$. By Lemma \ref{lemma: PBW belongs to V+} these are $$\EE_{\a}=q^{(\a,\la)}\E_{\a}, \ \FF_{\a}=\F_{\a}, \ \KK_{\a}=q^{(\a,\la)}k_{\a}$$
where $\E_{\a}, \F_{\a}, \k_{\a}$ are the root vectors in the old generators. In these new generators, the coproduct is given by
\begin{align*}
\De_{\la,\mu}(\EE_i)&=\EE_i\ot\ \KK_i+q^{-\la_i}1\ot \EE_i, & \De_{\la,\mu}(\FF_i)&=q^{\la_i}\KK_i^{-1}\ot \FF_i+\FF_i\ot 1, & \De_{\la,\mu}(\KK_i)&=\KK_i\ot \KK_i
\end{align*}
and the $R$-matrix becomes $R_{\la,\mu}=\KKK_{\la,\mu}\TT_{\la,\mu}$ with
$$\KKK_{\la,\mu}=\frac{1}{r^n}\sum_{\a,\b\in Q_r}q^{-(\a,\b)-(\a,\la)-(\b,\mu)}\KK_{\a}\ot \KK_{\b},$$
and
$$ \TT_{\la,\mu}=(\phi_{\la}\ot \id)(\TT_0)=\sum_I c_I\cdot q^{-(d(I),\la)}\cdot \phi_{\la}(\EE^I)\ot \FF^I=\sum_I c_I\cdot q^{(d(I),\la)}\cdot \EE^I\ot \FF^I.$$

\subsection{$\uAJ$ and the unrolled quantum group}
\label{subs: AJ and unrolled}
\def\AJ{A^J}
\def\uAJ{\uA^J}
\def\DeJ{\De^J}
\def\wDe{\De^J}
\def\RJ{R^J}
Though $\Ula$ is a quotient algebra of $\Urqg$ as seen above, its coproduct is different. We can remedy this by considering the Drinfeld twisted semi-direct product $\uAJ=\{\C[\hh^*]\ltimes \Ula\}_{\la\in\hh^*}$ of Subsection \ref{subs: TDDs}. Then the coproduct $\DeJ_{\la,\mu}:\AJ_{\la+\mu}\to \AJ_{\la}\ot \AJ_{\mu}$ becomes $$ \wDe_{\la,\mu}(\EE_i)=\EE_i\ot \KK_i+1\ot \EE_i, \ \wDe_{\la,\mu}(\FF_i)=\KK_i^{-1}\ot \FF_i+\FF_i\ot 1,\ \wDe_{\la,\mu}(\KK_i)=\KK_i\ot\KK_i.$$

Therefore, $\uA^J$ is isomorphic to the Hopf $\hh^*$-coalgebra obtained by setting $K_i^r=q^{r\la_i}$ in $\C[\hh^*]\ltimes \Urqg$ for each $i$. By (\ref{eq: R-matrix in AJ}), the $R$-matrix in $\uA^J$ is 
\begin{align}
   \RJ_{\la,\mu}=\phi_{\mu/2}\ot\phi_{\la/2}\cdot (\id\ot \phi_{\la/2})(R_{\la,\mu})=\wtKKK_{\la,\mu}\cdot \wt{\Theta}_{\la,\mu}
\end{align}
 where 
$$\wtKKK_{\la,\mu}=\phi_{\mu/2}\ot\phi_{\la/2}\cdot\KKK_{\la,\mu}, \ \ \wt{\Theta}_{\la,\mu}=(\id\ot\phi_{\la/2})(\TT_{\la,\mu})=\sum_I c_I \EE^I\ot\FF^I.$$

Note that conjugation by $\KK_i$ and $\phi_{\a_i/2}$ act the same way on $\AJ_{\la}$, hence we will mod out by $K_i=\phi_{\a_i/2}$ and still denote the quotient by $\AJ_{\la}$.

\begin{proposition}
\label{prop: A tilde contains unrolled}
The category of weight modules over $\UHr$ with weights in $\la+L_W$ embeds as the subcategory of $\Rep(\AJ_{\la})$ consisting of modules with diagonalizable $\hh^*$-action.
\end{proposition}
\begin{proof}
    Let $V\in\Rep(\AJ_{\la})$ and suppose the action of $\hh^*$ on $V$ is diagonalizable. For each $\mu\in\hh^*$ let $$V(\mu)=\{v\in V \ | \ \phi_z\cdot v=q^{2(\mu,z)}v \text{ for all } z\in\hh^*\}$$ so that $V=\oplus_{\mu\in\hh^*}V(\mu)$ (this holds because all characters $\hh^*\to\C^*$ can be written as $q^{2(\mu,z)}$ for some $\mu\in\hh^*$). Then $\KK_iv=q^{(\a_i,\mu)}v$ for all $v\in V(\mu)$ (since we mod out by $\KK_i=\phi_{\a_i/2}$). It is also easy to see that $\EE_i(V(\mu))\sb V(\mu+\a_i)$ and $\FF_i(V(\mu))\sb V(\mu-\a_i)$. Note also that since $\KK_i^r=q^{r\la_i}$ in $A_{\la}$ we must have $q^{r\la_i}=q^{r(\a_i,\mu)}$ for each $i$, which is equivalent to $\mu\in \la+L_W$. If for each $v\in V(\mu)$ we set $H_iv=(\a_i,\mu)v$, then $V$ becomes a weight module over $\UHr$ with weights in $\la+L_W$. Conversely, if $V$ is a weight $\UHr$-module with weights in $\la+L_W$, then $K_i^r$ acts by $q^{r\la_i}$ and on each weight space $V(\mu)$ we can define $\phi_z\cdot v=q^{2(\mu,z)}v$ for all $z\in\hh^*$. This defines a module structure over $\AJ_{\la}$ with diagonalizable $\hh^*$-action. 
\end{proof}

\def\UU{\mathcal{U}}
\def\UUa{\mathcal{U}_{\a}}
\def\CCC{\wt{\mathcal{C}}^H}
\def\CC{\mathcal{C}^H}

For each $\la\in\hh^*$, let $\CCC_{\la}$ be the category of weight modules over $\UHr$ with weights in $\la+L_R$. Let $\CCC=\coprod_{\la\in\hh^*}\CCC_{\la}$, then this is a ribbon category with ribbon structure induced from $\CC$. If $c_{V,W}$ denotes the braiding of $\CC$ (hence of $\CCC$), then $c^{\chi}_{V,W}=q^{(\la,\mu)}c_{V,W}$ is also a braiding in $\CCC$, where $V\in \CCC_{\la}, W\in\CCC_{\mu}$. The unframed link invariants are unaffected by using $c^{\chi}$.

 \def\x{\boldsymbol{x}} \def\y{\boldsymbol{y}}
\begin{proposition}
\label{prop: braided inclusion}
    The inclusion of $(\CCC, c^{\chi})$ into $\Rep(\uAJ)=\coprod_{\la\in\hh^*}\Rep(\AJ_{\la})$ is monoidal and braided.
\end{proposition}
\begin{proof}
    That the inclusion is monoidal follows because the formulas for $\De(\EE_i),\De(\FF_i),\De(\KK_i)$ are the same as the corresponding ones in $\UHr$. We now show that this inclusion is braided. Let $V\in\CCC_{\la}, W\in \CCC_{\mu}$ and restrict the braiding $c^J$ of $\Rep(\uAJ)$ to $V\ot W$. This is given by $$v\ot w\mapsto \tau(\RJ_{\la,\mu}(v\ot w))=\tau(\wtKKK_{\la,\mu}\cdot \wt{\Theta}_{\la,\mu}(v\ot w))$$
    and $\tau$ is the usual transposition. But $\wt{\Theta}_{\la,\mu}$ is exactly the same $\TT$ of Subsection \ref{subs: braiding on weight}. Thus, we only need to check that $\wtKKK_{\la,\mu}$ acts as $q^{(\la,\mu)}\HH$. But if $v\in V(\x)$ and $v\in W(\y)$ then
    \begin{align*}
        \wtKKK_{\la,\mu}(v\ot w)&=\phi_{\mu/2}\ot\phi_{\la/2}\cdot\frac{1}{r^n}\sum_{\a,\b\in Q_r}q^{-(\a,\b)-(\a,\la)-(\b,\mu)}\KK_{\a}\ot \KK_{\b}(v\ot w)\\
        &=q^{(\mu,\x)+(\la,\y)}\frac{1}{r^n}\sum_{\a,\b\in Q_r}q^{-(\a,\b)-(\a,\la)-(\b,\mu)+(\a,\x)+(\b,\y)}(v\ot w)\\
        &=q^{(\mu,\x)+(\la,\y)}\frac{1}{r^n}\sum_{\a\in Q_r}q^{(\a,\x-\la)} \left(\sum_{\b\in Q_r}q^{(-\a-\mu+\y,\b)}\right)(v\ot w)\\
         &=q^{(\mu,\x)+(\la,\y)}\frac{1}{r^n}\sum_{\a\in Q_r}q^{(\a,\x-\la)} r^n\d_{-\a-\mu+\y,0}(v\ot w)\\
         &=q^{(\mu,\x)+(\la,\y)+(\y-\mu,\x-\la)}(v\ot w)\\
         &=q^{(\x,\y)+(\la,\mu)}(v\ot w)\\
         &=q^{(\la,\mu)}\HH(v\ot w).
    \end{align*}
    Note that to get the fourth equality we used that $\mu-\y\in L_R$ (and that $r$ is comprime to $\det(A)$) so the inside sum in the third equality can be treated as in Lemma \ref{lemma: unrolled R-matrix restrict to ROSSO matrix in small quantum}.
   Thus, $c^J_{V,W}=q^{(\la,\mu)}c_{V,W}$, proving the proposition.
\end{proof}

\subsection{Ribbon structure}\label{subs: ribbon for TDDs} Let $\La_l$ be the left cointegral of $\uqb$ and $\mu_r\in \uqb^*$ be the right integral. Let $a\in \uqb, \a\in \uqb^*$ be the corresponding distinguished group-likes. A generalization of the usual Kauffman-Radford theorem (see \cite[Prop. 2.3]{LN:TDD}) states that ribbon structures on the twisted Drinfeld double $\underline{D(\uqb)}$ are classified by elements $b\in \uqb, \b\in\uqb^*$ and $p:\hh^*\to\C^*$ satisfying $b^2=a, \b^2=\a$, $\phi_{\la}(\La_l)=p(\la)^2\La_l$ for all $\la\in\hh^*$ and $S^2=\ad_{\b^{-1}}\circ\ad_b$. We call $(\b,b,p)$ a {\em Kauffman-Radford triple} (and $(\b,b)$ a Kauffman-Radford pair). The associated pivotal element is then given by $g_{\la}=p(\la)^{-1}\b\ot b$.\\

Let's find $(\b,b,p)$ for $\uqb$. The left cointegral is given by
\begin{align}
\label{eq: cointegral}
    \La_l=\left(\sum_{\c\in Q_r}\K_{\c}\right)\E^{I_{max}}
\end{align}
where $I_{max}$ is defined by $I_{max}(\b)=r'-1$ for all $\b\in\Dep$. The right integral is defined by $\mu_r(\K_{2\rho}^{1-r'}\E^{I_{max}})=1$ and $\mu_r=0$ on any other PBW basis element. From this, one can see the distinguished group-likes of $\uqb$ are given by
\begin{align*}
    a=\K_{2\rho}^{1-r'}, \ \a(\K_i)=q^{2}
\end{align*}
and $\a(\E_i)=0$ for $i=1,\dots,n$. Set 
\begin{align}
\b(\K_i)=q^{1-r'}, \b(\E_i)=0, \  b=\K_{2\rho}^{\frac{1-r'}{2}}=\K_{\rho}^{1-r'}, \ \ p(\la)=q^{(r'-1)(2\rho,\la)},
\end{align}
for all $i$, we claim that $(\b,b,p)$ is a Kauffman-Radford triple for $\uqb$. Clearly $\b^2=\a$ and $b^2=a$. We claim that $b\in \uqb$: if $r'$ is odd this is clear. If $r$ is even, since we assume $r$ coprime to $\det(A)$ it must be $\det(A)$ odd, but this only happens for $\sl_{n+1}$ with even $n$.  But then $$\K_{2\rho}=\prod_{i=1}^n\K_i^{i(n-i+1)}$$
and all powers of the $\K_i$'s are seen to be even. Hence $b\in \uqb$ as desired. One has $b\E_ib^{-1}=q^{(1-r')(\rho,\a_i)}\E_i=q^{1-r'}\E_i$ since $(2\rho,\a_i)=(\a_i,\a_i)$ for all $i$ (in the simply-laced case). Also $\ad_{\b^{-1}}(\E_i)=\b(\K_i)\E_i=q^{1-r'}\E_i$. Hence $S^2=\ad_{\b^{-1}}\circ \ad_b$ as desired. Moreover, one has $\phi_{\la}(\La_l)=q^{2(r'-1)(2\rho,\la)}\La_l$ hence  $\phi_{\la}(\La_l)=p(\la)^2\La_l$ as desired. \\

The $\K'_i$'s in $\uqb^*\sb D(\uqb)$ are characterized by $\lb \K'_i,\K_j\rb=q^{-a_{ij}}$. Hence, one has $\b=(\K'_{\rho})^{r'-1}$ in the $\K'_i$'s since $\lb k'_{\rho},k_j\rb=q^{-(\rho,\a_j)}=q^{-1}$. It follows that $\b$ maps to $\K_{\rho}^{1-r'}$ in $\Ula$ so that the pivot $g_{\la}=p(\la)^{-1}\b\ot b$ maps to $p(\la)^{-1}\K_{2\rho}^{1-r'}$. This is exactly $\KK_{2\rho}^{1-r'}$ in the rescaled generators. Passing to $\uA$ and then $\uAJ$ does not change the pivot (see \cite[Sect. 5]{LNV:genus}). Thus, we have shown that the embedding of Proposition \ref{prop: braided inclusion} preserves the pivotal structures. Since the ribbon structure is determined by the associated pivotal structure and viceversa, this implies:

\begin{proposition}
\label{prop: ribbon inclusion}
    The Kauffman-Radford triple $(\b,b,p)$ above induces a ribbon structure on $\ZZ_{rel}(\Rep(\uqb)\rtimes \hh^*)$ and its equivariantization that makes the embedding of Proposition \ref{prop: braided inclusion} a ribbon embedding.
\end{proposition}

\def\ep{\epsilon}

\subsection{Proof of Theorem \ref{theorem: equivariantization}}
\label{subs: proof of Theorem 2}
This follows from 
\begin{align*}
    (\CCC,c^{\chi})\hookrightarrow \Rep(\uAJ)&\cong \Rep(\uA)=\Rep(\underline{D(\uqb)}^{\hh^*})\\
  &\cong\Rep(\underline{D(\uqb)})^{\hh^*}\\
    &\cong\ZZ_{\Rep(\uqb)}(\Rep(\uqb)\rtimes\hh^*)^{\hh^*}.
\end{align*}

The first inclusion is Proposition \ref{prop: braided inclusion}, then we use Lemma \ref{lemma: Drinfeld twist gives equivalence}, Proposition \ref{prop: equivariantization is SS} and Proposition \ref{prop: rel Drinfeld and TDD} respectively. The part of the theorem concerning ribbon structures follows from Proposition \ref{prop: ribbon inclusion}.

\subsection{Recovering $P_{\uqb}^{\theta}(K)$} 
\label{subs: recovering P}
In \cite[Subs. 3.5]{LNV:genus}, given an oriented knot $K$, we defined a multivariable polynomial invariant by $$P_{\uqb}^{\theta}(K)=t_1^{w|\La_l|_1/2}\cdots t_n^{w|\La_l|_n/2} \ep_{D(\uqb')}(Z^{\theta}_{\underline{D(\uqb')}}(K_o)).$$ 
Here $K_o$ is a framed long knot whose closure is $K$, $Z^{\theta}_{\underline{D(\uqb')}}(K_o)$ is a ``twisted universal invariant" build from the twisted Drinfeld double of $\uqb'=\uqb\ot\C[t_1^{\pm 1},\dots, t_m^{\pm 1}]$ (the same TDD we defined above but where $t_i=q^{2\la_i}$), $\theta$ is the Hopf automorphism of $\uqb'$ defined by $\theta(\e_i)=t_ie_i, \theta(\K_i)=\K_i$ for every $i$, $w=w(K)$ is the writhe of $K_o$ and $(|\La_l|_1,\dots, |\La_l|_n)$ is the $\N^n$-degree of the cointegral $\La_l$ of $\uqb$. Here by twisted universal invariant we mean $Z^{\theta}_{\underline{D(\uqb')}}(K_o)$ takes care of the $G$-crossed structure of $\underline{D(\uqb')}$. However, this is computed from the usual universal invariant defined out of $\uA$ by 
\begin{align}
\label{eq: ZA is Z-TDD}
    Z_{\uA}^{\theta}(K_o)=Z^{\theta}_{\underline{D(\uqb')}}(K_o)\cdot \theta^w
\end{align}
see \cite[Subs. 3.3]{LNV:genus}.

\begin{proof}[Proof of Theorem \ref{corollary: P coincides with ADO}]
    
Set $t_i=q^{2\la_i}$ for every $i$, then $\theta=\phi_{\la}$. We will denote the universal invariants $Z^{\theta}|_{t_i=q^{2\la_i}}$ simply by $Z^{\la}$. Then $\underline{D(\uqb')}$ and $\uA$ are the same objects considered in this paper. By \cite[Lemma 5.3]{LNV:genus} we can use $Z_{\uAJ}^{\la}(K_o)$ instead of $Z_{\uA}^{\la}(K_o)$. Now, $V=V_{\la}$ is a module over $\AJ_{\la}$ and the $R$-matrix of $\uAJ$ acts as $c_{V,V}^{\chi}$ over $V_{\la}$ by Theorem \ref{theorem: equivariantization} and the $c_{V,V}^{\chi}$ invariant is the same as that using $c_{V,V}$ but multiplied by $q^{w(\la,\la)}$. Hence, for any $v\in V_{\la}$:
\begin{align}
\label{eq: ZAJ}
    Z_{\uAJ}^{\la}(K_o)v=q^{w(\la,\la)}\lb K_o\rb v.
\end{align}
Now, it is easy to see that $Z^{\la}_{\underline{D(\uqb')}}(K_o)v_0=\ep(Z^{\la}_{\underline{D(\uqb')}}(K_o))v_0$ and combining (\ref{eq: ZA is Z-TDD}) and (\ref{eq: ZAJ}) we get:
$$\ep(Z^{\la}_{\underline{D(\uqb')}}(K_o))v_0=Z_{\uA}^{\la}(K_o)\phi^{-w}_{\la}v_0= Z_{\uAJ}^{\la}(K_o)q^{-2w(\la,\la)}v_0=q^{-w(\la,\la)}\lb K_o\rb  v_0.$$
 From (\ref{eq: cointegral}) one sees that $|\La_l|_i=(r'-1)k_i$ for every $i$, where $2\rho=\sum_ik_i\a_i$. Hence $$(t_1^{|\La_l|_1}\cdots t_n^{|\La_l|_n})^{w/2}=q^{w(\la,(r'-1)2\rho)}$$ and thus
\begin{align}
    P_{\uqb}^{\theta}(K)|_{t_i=q^{2\la_i}}=q^{w(\la,(r'-1)2\rho)-w(\la,\la)} \lb K_o\rb=q^{-w(\la,\la-(r'-1)2\rho)}\lb K_o\rb
\end{align}
which is exactly the unframed invariant $\ADOgr'(K)$ as in (\ref{eq: unframed ADO}).

\end{proof}

\bibliographystyle{amsplain}
\bibliography{/Users/daniel/Desktop/Daniel/TEX/bib/referencesabr}

\end{document}